\documentclass[a4paper,10pt]{amsart}
\usepackage{amsfonts,amssymb,amsthm, mathrsfs, lscape}
\usepackage{verbatim}

\usepackage{latexsym}

\usepackage{graphics}
\usepackage{graphicx}
\usepackage[space]{cite}

\usepackage{latexsym}
\usepackage{amsfonts}
\usepackage{amsmath}

\usepackage{rawfonts}

 \usepackage[usenames,dvipsnames]{pstricks}
 \usepackage{epsfig}
 \usepackage{pst-grad} 
 \usepackage{pst-plot} 
 \usepackage[space]{grffile} 
 \usepackage{etoolbox} 
 \makeatletter 

\renewcommand{\epsilon}{\varepsilon}

\newcommand{\R}{{\mathbb R}}

\renewcommand{\phi}{\varphi}

\newcommand{\hcal}{\mathcal{H}}

\newtheorem*{theono}{Theorem}
\newtheorem{theorem}{{Theorem}}[section]
\newtheorem{theo}{{Theorem}}[section]
\newtheorem{cor}[theorem]{{Corollary}}
\newtheorem{lem}[theorem]{{Lemma}}
\newtheorem{prop}[theorem]{{Proposition}}

\newenvironment{rem}{\medskip\noindent{\it Remark:\/} }{\medskip}

\newtheorem{defin}[theo]{{\sc Definition}}

\newtheorem{conj}[theorem]{Conjecture}
\newtheorem*{conj*}{Conjecture}

\newtheorem{rmk}[theorem]{Remark}

\theoremstyle{definition}

\numberwithin{equation}{section}

\def \R {\mathbb R}

\title[Balanced embedding problem III:Asymptotics]{ An infinite dimensional balanced embedding problem III: Asymptotics near infinity}

\author{Jingzhou Sun}
\thanks{The author is partially supported by NNSF of China no.11701353.}
\address{Department of Mathematics, Shantou University, Shantou City, Guangdong Province 515063, China}
\email{jzsun@stu.edu.cn}

\begin{document}

\begin{abstract}
We continue our study on the logarithmic balanced model metric initiated in our previous work. By a non-trivial refinement of the set of tools developed in our previous work, we are able to confirm partially a conjecture we made in our previous work on the asymptotic behavior of the balanced metric near infinity.
\end{abstract}

\maketitle

 \tableofcontents

\section{Introduction}
In our previous work \cite{sun-sun-2}, we initiated the investigation on the logarithmic balanced model metrics. We refer the readers to \cite{sun-sun-2} for the motivations for this series of research.

Let $\mathcal H$ be a separable complex Hilbert space, and let $\mathbb P(\mathcal H)$ be the associated projective Hilbert space, viewed as a set. We denote by $\mathfrak h(\mathcal H)$ the space of bounded self-adjoint operators on $\mathcal H$. There exists a natural injective map 
	$$\iota: \mathbb P(\mathcal H)\rightarrow \mathfrak h(\mathcal H); [z]\mapsto \frac{z\otimes z^*}{|z|^2},$$
	where $z\otimes z^*$ is the operator defined by $y\mapsto \langle y,z\rangle z,y\in \hcal$.
Consider a pair $(X, D)$ of complex manifolds, where $D$ is a divisor in $X$. Let $d\mu$ be a volume form on $X$ and $d\nu$ be a volume form on $D$. We say that a holomorphic embedding $F: X\rightarrow \mathbb P(\mathcal H)$ is $(\beta, C)$-balanced for some $\beta, C\in \mathbb R$ if
	$$ \int_X \iota(F(z)) d\mu+(1-\beta) \int_D \iota(F(z)) d\nu=C\cdot\text{Id},$$
	where $\iota(F(z))$ denotes the K\"ahler form on $\mathbb P(\mathcal H)$ induced by the Fubini-Study metric. 

  When $X$ is non-compact, this problem is of infinite-dimensional nature and very challenging.

In the case where $X$ is the complex plain and $D$ is the origin, we have proved the existence and uniqueness of the balanced model metric.
\begin{theono}[\cite{sun-sun-2}]
  For $\beta\in [0, 1)$ there is a $(\beta ,1)$-balanced embedding $F: X\rightarrow \mathbb P(\mathcal H)$ which is unique up to the action of $U(\mathcal H)$. 
\end{theono}
Such a balanced metric corresponds to an analytic function $f(x)=\sum_{i=0}^{\infty}c_i x^i $ satifying 
\begin{equation}\label{e-1}
	\int_0^{\infty} \frac{f(sx)}{f(x)}dx=\frac{1}{1-s}-\beta, 
	\end{equation}
	for all $s\in[0,1)$. 

 It is then both natural and important to study the asymptotics of $f(x)$ near infinity. A good understanding of the asymptotics can help us study more general balanced embedding problems and help us to understand the connection to the balanced embedding problem in the compact setting(\cite{Donaldson10, Sun, Sun2019, SS}). 
In \cite{sun-sun-2}, we made a conjecture about the asymptotic of $f(x)$ near infinity:
\begin{conj}\label{conj}
	Let $f(x)$ be the unique solution to \eqref{e-1} with $f(0)=1$, then $f$ should asymptotically be $C_\beta x^\beta e^x$ as $x\to +\infty$.
\end{conj}

A weaker version of this conjecture is that $$\lim_{x\to \infty}\frac{\log f(x)-x}{\log x}=\beta.$$
The main result of this article is the following theorem, in the statement of which we write $f_\beta$ in place of $f$ to emphasize the dependence of $f$ on $\beta$.
\begin{theorem}\label{thm-main}
  $$\omega (\beta)= \lim_{x\to \infty}\frac{\log f_\beta(x)-x}{\log x}$$ exists for $\beta\in [0,1).$
\end{theorem}
The main idea of the proof is to study the asymptotics of $\frac{d^2}{dt^2}\log f(x),$ where $t=\log x$. We can modify the arguments in the proof of lemma 4.2 in \cite{sun-sun-2} to get the following.
\begin{theorem}\label{thm-second}
  $\omega (\beta)$ is an increasing function, and for each $\beta\in [0, 1)$ we have 
  $$\omega (\beta)\leq 1.$$
\end{theorem}
Unfortunately, we are currently not able to show if $\omega(\beta)$ is continuous or strictly increasing. We hope to be able to do this in our future work.

The proof of theorem \ref{thm-main} also implies, as we will explain after the proof of theorem \ref{thm-main}, that 
$$C_\beta=\lim_{x\to\infty }\frac{f(x)}{x^{\omega(\beta)}e^x}  \textit{ exists}.$$
It seems that we need some new ideas to calculate the value of $C_\beta$.

\

We now talk about the proof of our main result. The main strategy in proving theorem \ref{thm-main} is to prove an asymptotics of $\frac{d^2}{dt^2}\log f(x)$ with very small error (corollary \ref{cor-final}). Recall that in \cite{sun-sun-2} we developed a set of tools and showed that $\lim_{x\to \infty}\frac{1}{x}\frac{d^2}{dt^2}\log f(x)=1$. The machinery there was to narrow down three things simultaneously. We are able to sharpen this set of tools to prove a much finer asymptotic of $\frac{d^2}{dt^2}\log f(x)$. In doing so, there are three major difficulties.

 The first one is the error in estimating summations with integrals. The naive estimate of this error is too large for theorem \ref{thm-main}. Our idea is to use the Poisson summation formula to estimate this error. But in order to use Poisson summation formula, we need good estimates of the derivatives of $\log f(x)$. So we have to use the naive estimate first and run our machine to get estimates of the derivatives of $\log f(x)$ which happen to be good enough for using the Poisson summation formula. 
 
 The second major difficulty is the internal error in the machinery used in \cite{sun-sun-2}. More precisely, for large $a$, we use the integrals on the interval $I_x=[x-\sqrt{x}\log x,x+\sqrt{x}\log x ]$ to estimate the integrals over $[0,\infty)$ or $(-\infty,\infty)$. But if we have bounds like $$M_1<\frac{1}{x}\frac{d^2}{dt^2}\log f(x)<M_2,$$
 then on the interval $I_0$, we would have $$M_1x_0(1-\frac{\log x_0}{\sqrt{x_0}})<\frac{d^2}{dt^2}\log f(x)<M_2x_0(1+\frac{\log x_0}{\sqrt{x_0}}).$$
The ratio of the upper and lower bounds will carry an extra $O(\frac{\log x_0}{\sqrt{x_0}})$ term, making a barrier to getting asymptotics finer than $\frac{\log x}{\sqrt{x}}$. To overcome this difficulty, we developed a method in section \ref{sec-looseends} to replace the original functions with new ones.

The third major difficulty or complication is the mutual dependence of the estimates of different quantities. We (only)have to run the machine three times, each time producing errors small enough for the next round.

\

The structure of this article is as follows. We first recall the results from \cite{sun-sun-2} necessary for this article. Then in section \ref{sec-3} we run the machine for the first time. Then in section \ref{sec-looseends}, we deal with the two loose ends whose contribution of errors is too large. Then in section \ref{sec-5}, we run the machine again and use the new smaller error to take care of one loose end again. Then in section \ref{sec-6}, we use the new smaller error to reestimate all the relative terms and get prepared for the third round. In the last section, we run the machine for the third time and get an estimate fine enough for theorem \ref{thm-main}. At the end of the last section, we will prove theorem \ref{thm-second}.

\

\textbf{Acknowledgements.} The author would like to thank Professor Song Sun for many helpful discussions.

\section{Setting-up}
\subsection*{Notations}
\begin{itemize}
  \item We will denote by $\epsilon(a)$ a term that is asymptotically smaller than $a^{-n}$ for any $n>0$.
  \item Denote $t=\log x$ and $$u=\frac{d}{dt}\log f(x)=x\frac{f'(x)}{f(x)}.$$
  \item We extend the coeffients $c_i$ to a coeffient function of real $a> 0$: \begin{equation*}c(a)=(\int_0^\infty \frac{x^a}{f(x)}dx)^{-1} \label{e:definition of ca}	
	\end{equation*}
  \item We define $\lambda(a)=-\log c(a)$. When $a>0$ is an integer, we will write $\lambda_a=\lambda(a)$
  \item We normalize $f(0)=1$.
  \item For each $a>0$, we denote by $x_a$ the point so that $u(x_a)=a$. We also denote by $t_a=\log x_a$.
  \item For each $x>0$, we denote by $n_x$ the point so that $\lambda'(n_x)=x$.
  \item We denote by $\tilde{a}=n_{x_a}$.
  \item For $a\geq1$ we denote $$h_a(x)=\frac{f(x)}{c(a)x^a}.$$
  \item We define $\nu(a)=\frac{h_a(x_a)}{\sqrt{x_a}}$.
  \item We denote by $\tilde{t}_a=\lambda''(a)$.
  \item We denote by $\Delta_a=\frac{c(\tilde{a})}{c(a)}x_a^{n_{x_a}-a}$.
  \item We denote by $g_a(t)=\log f(x)-(n+1)t-(\log f(x_{a+1})-(n+1)t_{a+1})$ so that we have $$g_a(t_{a+1})=g_a'(t_{a+1})=0.$$
  \item We denote by $G_a(i)=\lambda(i)-\lambda(\tilde{a})+(\tilde{a}-i) \log x_a$ so that we have $$G_a(\tilde{a})=G'_a(\tilde{a})=0.$$
\end{itemize}
\begin{rmk}\label{rmk-1}
  \begin{itemize}
    \item Notice that for a integer $i>0$, we have $c_i=c(i)$, but $c_0\neq c(0)$. 
    \item  We have that $g_a(t)$ is a strictly convex function of $t$ whose unique minimum is attained at $t_{a+1}=\log x_{a+1}$. 
    \item   We also have that $G_a(i)$ is a strictly convex function of $i$ whose unique minimum is attained at $\tilde{a}$.
    \item Theorem \ref{thm-start} then implies that the mass of $e^{-g_a(t)}$ and $e^{-G_a(i)}$ are concentrated around the points $t_{a+1}$ and ${\tilde{a}}$ respectively. More precisely, we have $$\int_{|t-t_{a+1}|>\frac{\log a}{\sqrt{a}}}(t-t_{a+1})^ke^{-g_a(t)}dt=\epsilon(a),$$ and $$\sum_{|i-\tilde{a}|>\sqrt{a}\log a}(i-\tilde{a})^ke^{-G_a(i)}=\epsilon(a),$$ for all $k\geq 0$.
  \end{itemize}
\end{rmk}

\subsection{preparation}
In the following, we recall the results from \cite{sun-sun-2} necessary for this article.
We first recall the following theorem.
\begin{theorem}[theorem 3.19 in \cite{sun-sun-2}]\label{thm-start}
  We have the following:
		\begin{equation}
		\lim_{a\to \infty}\frac{h_a(x_a)}{\sqrt{a}}=\sqrt{2\pi},
		\end{equation}
		\begin{equation}
		\lim_{a\to \infty}a\frac{d^2}{da^2}\lambda(a)=1,
		\end{equation}
		and
		\begin{equation}
		\lim_{x\to \infty} \frac{\frac{d^2}{dt^2}\log f(x)}{x}=1
		\end{equation}
\end{theorem}
This theorem allows us to define the following functions for $b>1$.
Let $$M_1'(b)=\inf\{\frac{\frac{d^2}{dt^2}\log f(x)}{x}|x>x_b-\sqrt{b}\log b \} ,$$
$$M_2'(b)=\sup \{\frac{\frac{d^2}{dt^2}\log f(x)}{x}|x>x_b-\sqrt{b}\log b \},$$
$$M_1''(b)=\inf\{x_a\lambda''(a)|a>b-\sqrt{b}\log b \}, $$and 
$$M_2''(b)=\sup \{x_a\lambda''(a)|a>b-\sqrt{b}\log b \} .$$

\begin{defin}
  For a pair of numbers $0<M_1<M_2<\infty $, we call a $C^{2}$ function $g(t)$ on $\R$ be a $(M_1,M_2)$-admissible function if it  satisfies the following:
		\begin{itemize}
			\item[(1)]$g(0)=0$;
			\item[(2)]$g'(0)=0$;
			\item[(3)] there exist  such that $M_1<g''(t)<M_2$ for all  $t\in \R $.
		\end{itemize}
\end{defin}

    We define the following functions of $g$: 
    $$\bar{t}(g)=\frac{\int_{-\infty}^{\infty }te^{-g(t)}dt}{\int_{-\infty}^{\infty }e^{-g(t)}dt},$$
    $$A(g)=\int_{-\infty}^{\infty }e^{-g(t)}dt,$$
    $$B(g)=\int_{-\infty}^{\infty }(t-\bar{t})^2e^{-g(t)}dt,$$
    and 
    $$\tilde{d}(g)=\frac{B(g)}{(A(g))^3}. $$
    \begin{rmk}
      Notice that $\tilde{d}(g)$ does not change if we dilate $t$, namely $\tilde{d}(g(t))=\tilde{d}(g(\lambda t))$. 
    \end{rmk}

    Now for fixed $m>0$, we defined the function 
    $$d_m(c)=\frac{\int_0^\infty y^2e^{-l_c(y)}dy}{(\int_0^\infty e^{-l_c(y)}dy)^3},$$
    where $l_c(y)$ satisfies:
    \begin{itemize}
      \item $l_c(y)=\frac{y^2}{2}$ for $y\leq c$;
      \item $l_c(y)=\frac{my^2}{2}+(1-m)cy+\frac{mc^2}{2}-\frac{c^2}{2}$ for $y>c$.
    \end{itemize}
When $m<1$, $d_m(c)$ attains its maximum for some $c>0$, and when $m>1$, $d_m(c)$ attains its minimum for some $c>0$.
For $m\geq 1$, we denote $$p(m)=\frac{1}{4}\max_{c\geq 0}d_{1/m}(c),$$ and $$q(m)=\frac{1}{4}\min_{c\geq 0}d_{m}(c).$$So we can define $$F(m)=(\frac{p(m)}{q(m)})^2.$$ 
Then we define $$\tilde{p}(m)=4p(m)\frac{\alpha^3+1}{(1+\alpha)^3},$$ where $\alpha=\frac{1}{\sqrt{m}}$, and let $$\tilde{q}(m)=q(m)-\frac{4}{\pi^2}\frac{(1-\alpha)^2}{(1+\alpha)^2} .$$ 
We then have the following theorem.
\begin{theorem}[corollaries 3.15 and 3.16 combined in \cite{sun-sun-2}]
  Let $g(t)$ be a $(M_1,M_2)$-admissible function. Let $m=\frac{M_2}{M_1}$. Then we have
  $$\tilde{q}(m)\leq \tilde{d}(g)\leq \tilde{p}(m) $$
\end{theorem}
In order to show the convergence of our later iterations, we need the following results from subsection 3.3 in \cite{sun-sun-2}:
\begin{itemize}
  \item $\lim_{m\to 1}q(m)=\lim_{m\to 1}p(m)=\frac{1}{2\pi},$
  \item When $1\leq m <1.01$, we have $\frac{p(m)}{q(m)}<1.005$, $0.031<p'(m)<0.033$, $-0.033<q'(m)<-0.031, $ and $0<F'(m)<0.84$;
  \item When $1\leq m <1.01$, $1\leq \Delta_a<e^{\frac{1.1}{\pi}(1-\frac{1}{\sqrt{m}})^2}.$
  \item Let $P(m)=4\frac{\alpha^3+1}{(1+\alpha)^3}$, then for $1\leq m <1.01$, we have $1\leq P(m)<1.00002$ and $0\leq P'(m)<0.0037.$
\end{itemize}
Therefore, we have 
$-0.033<\tilde{q}'(m)<-0.03$, $0.031<\tilde{p}'(m)<0.034,$ 
and $$0\leq \frac{d}{dm}(\frac{\tilde{p}^2(m)}{\tilde{q}^2(m)})<0.89,$$
which imply 
$$0.031(m-1)<\frac{1}{2\pi}-\tilde{q}(m)<0.03(m-1) ,$$ and $$0.031(m-1)<\tilde{p}(m)-\frac{1}{2\pi}<0.034(m-1),$$
and 
for $1\leq m <1.01$.

It is not hard to get the following formula for $\lambda$,
\begin{equation}\label{e-lam2}
  \lambda''(a)=\int_0^\infty\frac{(\log x-\tilde t_a)^2dx}{h_a(x)}
\end{equation}
For $\log f(x)$, we have 
\begin{equation}\label{e-f2}
  \frac{d^2}{dt^2}\log f(x_a)=\frac{d^2}{dt^2}\log h_a(x_a)=\sum(i-a)^2\tau_{x_a}(i),
\end{equation}
	where \begin{eqnarray*}
		\tau_{x_a}(i)=\frac{1}{h_a(x_a)}\frac{c_i}{c_a}x_a^{i-a}=\frac{\Delta_a}{h_a(x_a)}e^{-G_a(i)}.
	\end{eqnarray*}
 We also have \begin{equation}\label{e-f3}
  \frac{d^3}{dt^3}\log f(x_a)=\frac{d^3}{dt^3}\log h_a(x_a)=\sum(i-a)^3\tau_{x_a}(i).
\end{equation}

\section{First level narrowing}\label{sec-3}
Since $a=\frac{d}{dt}\log f(x_a)$, $\frac{d}{dt_a}a=\frac{d^2}{dt^2}\log f(x_a)$.  Then $$\frac{d}{da}\log \nu(a)=\tilde{t}_a-t_a-\frac{1}{2\frac{d^2}{dt^2}\log f(x_a)},$$ and $$\frac{d^2}{da^a}\log \nu(a)=\lambda''(a)-\frac{1}{\frac{d^2}{dt^2}\log f(x_a)}+\frac{\frac{d^3}{dt^3}\log f(x_a)}{2(\frac{d^2}{dt^2}\log f(x_a))^3}.$$

Since $\lim_{a\to \infty}a\lambda''(a)=1$, we have
$$\frac{d^3}{dt^3}\log f(x)=O(x^{3/2}).$$

 Then let $M_1(b)=\min(M_1'(b),M_1''(b))$, $M_2(b)=\min(M_2'(b),M_2''(b))$.
On the interval $P_b=[x_b-\sqrt{b}\log b,x_b+\sqrt{b}\log b]$, we have $$M_1(x_b-\sqrt{b}\log b)\leq \frac{d^2}{dt^2}\log f(x)\leq M_2(x_b+\sqrt{b}\log b)$$
 Similarly, on the interval $Q_b=[b-\sqrt{b}\log b,b+\sqrt{b}\log b]$, 
 $$\frac{M_1}{x_b+\sqrt{b}\log b}\leq \lambda''(a)\leq \frac{M_2}{x_b-\sqrt{b}\log b}.$$
 So we let $$m(b)=(1+3\frac{\log b}{\sqrt{b}})\frac{M_2(b)}{M_1(b)}.$$ 
 \begin{rmk}
  Theorem \ref{thm-start} implies that $\tilde{a}-a=o(\sqrt{a})$, so the mass concentration property of $e^{-G_a(i)}$ mentioned in remark \ref{rmk-1}
  is still valid for the interval $Q_b$, which is more convenient.
 \end{rmk}

 Then Let $p_1=\tilde{p}(m(b))$ and $q_1=\tilde{q}(m(b))$.
 Then we have
 \begin{equation}\label{e-l-general}
  q_1(1-O(\frac{1}{b}))\frac{h^2_b(x_b)}{x_b^2}\leq \lambda''(b)\leq p_1(1+O(\frac{1}{b}))\frac{h^2_b(x_b)}{x_b^2},
 \end{equation}

 and
 \begin{equation}\label{e-f-general}
  \frac{1}{\Delta_b^2}q_1(1-O(\frac{1}{\sqrt{b}}))h^2_b(x_b)\leq \frac{d^2}{dt^2}\log f(x_b)\leq \frac{1}{\Delta_b^2}p_1(1+O(\frac{1}{\sqrt{b}}))h^2_b(x_b).
 \end{equation}
\subsection{non-monotonic case}
In this subsection, we assume that $\nexists n>0$ such that $\nu(a)$ is monotone for $a>n$. This assumption applies to all the results in this subsection.

 Let $b$ be a local minimum point of $\nu(a)$. By assumption, there are infinite such $b$, so our following notations $O(\frac{1}{\sqrt{b}})$, etc. make sense. Then we have 
 $$\tilde{t}_b-t_b=\frac{1}{2\frac{d^2}{dt^2}\log f(x_b)},$$
 and 
$$\lambda''(b)-\frac{1}{\frac{d^2}{dt^2}\log f(x_b)}+\frac{\frac{d^3}{dt^3}\log f(x_b)}{2(\frac{d^2}{dt^2}\log f(x_b))^3}\geq 0.$$
So $$\lambda''(a)\frac{d^2}{dt^2}\log f(x_b)\geq 1-O(\frac{1}{\sqrt{x_b}}).$$
Since $t_{a+1}-t_a=\frac{1+o(1)}{x_a}$, we have $x_{b+1}=x_b+1+o(1)$ and $$\tilde{t}_b-t_{b+1}=-\frac{1+o(1)}{2x_b}.$$
So $\frac{d}{di}g_a(i)=\frac{1}{2\frac{d^2}{dt^2}\log f(x_a)}$ at $i=a=b$, implying that $b-n_{x_b}=\frac{1}{2}+o(1)$. So $\Delta_b=e^{O(\frac{1}{a})}=1+O(\frac{1}{a})$. 
Therefore
$$h_b(x_b)=(1+O(\frac{1}{\sqrt{a}}))\int_{0}^{\infty }e^{-G_a(i)}di.$$
Since $h_b'(x_b)=0$, $h_b(x_{a+1})=(1+O(\frac{1}{x_b}))h_b(x_{b})$, therefore
$$\frac{h_b(x_b)}{x_b}=(1+O(\frac{1}{a}))\int_{-\infty}^{\infty }e^{-g_a (t)}dt.$$

 Then
 $$\frac{d^2}{dt^2}\log f(x_b)\leq p_1(1+O(\frac{1}{\sqrt{b}}))h^2_b(x_b),$$ and
 $$\lambda''(b)\leq p_1(1+O(\frac{1}{b}))\frac{h^2_b(x_b)}{x_b^2}.$$

 So we get that 
$(1+O(\frac{1}{\sqrt{b}}))(\nu(b))^4p_1^2\geq 1$. So $$\nu(b)>(1-\frac{C_1}{\sqrt{b}})\frac{1}{\sqrt{p_1}}$$ for $b$ large enough, for some $C_1>0$ independent of $b$.
And similarly, let $b'>b$ be a local maximum, then we have $$\nu(b')<(1+\frac{C_2}{\sqrt{b}})\frac{1}{\sqrt{q_1}},$$for some $C_2>0$ independent of $b$. 
Then for any $a>b$, $\exists b',b''$ such that 
\begin{itemize}
  \item $b\leq b''\leq a\leq b'$;
  \item $b''$ is a local minimum point and $b'$ is a local maximum point;
  \item $\nu(b'')\leq \nu(a)\leq \nu(b')$.
\end{itemize}
So we have
$$(1-\frac{C_1}{\sqrt{b''}})\frac{1}{\sqrt{p_1}}< \nu(a)< (1+\frac{C_2}{\sqrt{b'}})\frac{1}{\sqrt{q_1}}.$$
Therefore, we have
$$(1-\frac{C_1}{\sqrt{b}})\frac{1}{\sqrt{p_1}}< \nu(a)< (1+\frac{C_2}{\sqrt{b}})\frac{1}{\sqrt{q_1}},$$
for all $a\geq b$.

Let $C_3=\max\{C_1,C_2\}$. 
\begin{lem}\label{lem-p1q1}
  For $a$ large enough, assume $m(a)>1+30\frac{\log a}{\sqrt{a}}$, then
  $$(1-\frac{C_3}{\sqrt{a}})\frac{1}{\sqrt{p_1}}< \nu(a')< (1+\frac{C_3}{\sqrt{a}})\frac{1}{\sqrt{q_1}},$$
  where $p_1=p(m(a))$ and $q_1=q(m(a))$,
  for all $a'\geq a+2\sqrt{a}\log a$.
\end{lem}
\begin{proof}
  We have proved the case when $a$ is a local extremal point of $\nu$. Therefore if there is a local extremal point of $\nu$ on the interval $[a,a+2\sqrt{a}\log a]$, the conclusion of the proposition also follows. So we only need to consider the case when $\mu$ is monotone on $[a,a+2\sqrt{a}\log a]$. We assume that $\mu$ is decreasing. Then $\exists a_1>a+2\sqrt{a}\log a$ a local minimum point of $\nu$ such that $\nu(a_1)\leq \nu(a')$ for all $a'\in [a,a_1]$. It then follows that we have $$\nu(a')\geq (1-\frac{C_3}{\sqrt{a}})\frac{1}{\sqrt{p_1}},$$ for all $a'\geq a$. If $\nu(a+2\sqrt{a}\log a)< (1+\frac{C_3}{\sqrt{a}})\frac{1}{\sqrt{q_1}}$, then the upper bound for $a'\geq a+2\sqrt{a}\log a$ also follows. So we assume that $\nu(a+2\sqrt{a}\log a)\geq (1+\frac{C_3}{\sqrt{a}})\frac{1}{\sqrt{q_1}}$, hence $$\nu(a')\geq (1+\frac{C_3}{\sqrt{a}})\frac{1}{\sqrt{q_1}},$$ for $a'\in [a,a+2\sqrt{a}\log a]$.

  Notice that by formula \ref{e-l-general}, for all $b\geq a$ we have $$\lambda''(b)\geq q_1(1-\frac{C_4}{b})\frac{h^2_b(x_{b})}{x_{b}^2},$$
  for some $C_4$ independent of $a'$. Therefore 
  $$x_b\lambda''(b)\geq (1-\frac{C_7}{b})(1+\frac{C_3}{\sqrt{a}})^2,$$
  for $b\in [a,a+2\sqrt{a}\log a]$.

  Let $a_2=a+\sqrt{a}\log a$, then 
  $$x_{a_2}\lambda''(b)\geq (1-1.5\frac{\log a}{\sqrt{a}}),$$
  for $b\in [a,a+2\sqrt{a}\log a]$.
  So we have 
  $$\nu(a_2)\leq \Delta_{a_2}(1+\frac{\log a}{\sqrt{a}})\sqrt{2\pi}.$$
  Since $\Delta_a<e^{\frac{1.1}{\pi}(1-\frac{1}{\sqrt{m}})^2}$, we get 
  \begin{eqnarray*}
    (1+\frac{C_3}{\sqrt{a}})\frac{1}{\sqrt{q_1}}\leq e^{\frac{1.1}{\pi}(1-\frac{1}{\sqrt{m}})^2}(1+\frac{\log a}{\sqrt{a}})\sqrt{2\pi}.
  \end{eqnarray*} 
  Since $1-\alpha\leq \frac{1}{2}(m-1)$, $$e^{-\frac{1.1}{\pi}(1-\frac{1}{\sqrt{m}})^2}-1\geq -\frac{1.1}{4\pi}(m-1)^2.$$
Therefore \begin{eqnarray*}
  \sqrt{2\pi q_1}&\geq& e^{-\frac{1.1}{\pi}(1-\frac{1}{\sqrt{m}})^2}(1-\frac{\log a}{\sqrt{a}})\\
  &\geq&(1-0.01(m-1)^2)(1-\frac{\log a}{\sqrt{a}}).
\end{eqnarray*}
The assumption that $m(a)>1+\frac{\log a}{\sqrt{a}}$ implies that $2\pi q_1<1-0.09(m-1)$.Therefore
$$1-0.05(m-1)>(1-0.01(m-1)^2)(1-\frac{\log a}{\sqrt{a}}),$$
which is impossible if $m-1>30\frac{\log a}{\sqrt{a}}.$

The case when $\mu$ is increasing is similar and simpler.
\end{proof}
\begin{theorem}\label{theo-non-monotone}
  For $a$ large engough, we have $$m(a)\leq 1+\frac{114\log a}{\sqrt{a}}.$$
\end{theorem}
\begin{proof}
  It is enough to show that for $a$ large engough, $$m(2a)\leq 1+\frac{80\log a}{\sqrt{a}}.$$ 
  We can take $a$ large enough so that $m(a)<1.01$.
Assume this is not the case, then $m(a')>1+\frac{80\log a}{\sqrt{a}},$
for $a'\in [a,2a]$. Therefore $$m(a')>1+\frac{80\log a'}{\sqrt{a'}},$$ for $a'\in [a,2a]$. 
So we can apply the conclusion of lemma \ref{lem-p1q1}. 
So let $p_1=p(m(a))$ and $q_1=q(m(a))$, then we get 
$$(1-\frac{C_3}{\sqrt{a}})\frac{1}{\sqrt{p_1}}< \nu(a')< (1+\frac{C_3}{\sqrt{a}})\frac{1}{\sqrt{q_1}},$$
for all $a'\geq a_1=a+2\sqrt{a}\log a$. Since 
\begin{equation*}
  q_1(1-O(\frac{1}{b}))\frac{h^2_b(x_b)}{x_b}\leq x_b\lambda''(b)\leq p_1(1+O(\frac{1}{b}))\frac{h^2_b(x_b)}{x_b},
 \end{equation*}
 and
 \begin{equation*}
  e^{-\frac{2.2}{\pi}(1-\frac{1}{\sqrt{m(a)}})^2}q_1(1-O(\frac{1}{\sqrt{b}}))h^2_b(x_b)\leq \frac{d^2}{dt^2}\log f(x_b)\leq p_1(1+O(\frac{1}{\sqrt{b}}))h^2_b(x_b),
 \end{equation*}
for $b\geq a_1$, we get
\begin{equation*}
  (1-\frac{C_5}{\sqrt{a}})\frac{q_1}{p_1}\leq \lambda''(b)\leq (1+\frac{C_5}{\sqrt{a}})^2\frac{p_1}{q_1},
 \end{equation*}
 and
 \begin{equation*}
  e^{-\frac{2.2}{\pi}(1-\frac{1}{\sqrt{m(a)}})^2}(1-\frac{C_6}{\sqrt{a}})\frac{q_1}{p_1}\leq \frac{1}{x_b}\frac{d^2}{dt^2}\log f(x_b)\leq (1+\frac{C_6}{\sqrt{a}})\frac{p_1}{q_1},
 \end{equation*}
for $b\geq a_1$, for some $0<C_5\leq C_6$ independent of $a$.

So let $a_1'$ satisfy $a_1'-\log a_1'\sqrt{a_1'}=a_1$, then $$M_1(a_1')>e^{-\frac{2.2}{\pi}(1-\frac{1}{\sqrt{m(a)}})^2}q_1(1-\frac{C_6}{\sqrt{a}})$$ and $$M_2(a_1')<p_1(1+\frac{C_6}{\sqrt{a}}).$$
 Therefore, \begin{eqnarray*}
  m(a_1')&<&(1+3\frac{\log a_1'}{\sqrt{a_1'}})e^{-\frac{2.2}{\pi}(1-\frac{1}{\sqrt{m(a)}})^2}\frac{p^2_1}{q^2_1}(1+\frac{3C_6}{\sqrt{a}})\\
  &<&(1+3.5\frac{\log a_1'}{\sqrt{a_1'}})e^{-\frac{2.2}{\pi}(1-\frac{1}{\sqrt{m(a)}})^2}\frac{p^2_1}{q^2_1} .
 \end{eqnarray*}

 $$m(a_1')-1<0.9(m(a)-1)(1+3.5\frac{\log a_1'}{\sqrt{a_1'}})+3.5\frac{\log a_1'}{\sqrt{a_1'}}<0.95(m(a)-1).$$
 Then for $n\geq 1$ we define $a_{n+1}=a_n'+2\sqrt{a_n'}\log a_n'$ and $a_{n+1}'$ accordingly and inductively. And by induction, we get 
 $$m(a_n')-1<0.95^{n-1}(m(a)-1).$$
 It is easy to see that when $n=\frac{\sqrt{a}}{4\log a}+1$, $a_{n}'<2a$. But $$0.95^{\frac{\sqrt{a}}{4\log a}}<\frac{\log a}{\sqrt{a}},$$
 So we have arrived at a contraction.
\end{proof}
\subsection{the general case}
\begin{theorem}\label{theo-general-log}
  For $a$ large engough, we have $$m(a)\leq 1+\frac{114\log a}{\sqrt{a}}.$$
\end{theorem}
\begin{proof}
  The non-monotonic case has been proved in theorem \ref{theo-non-monotone}. If $\nu(a)$ is monotone, we assume, for example, 
  that $\nu'(a)\geq 0$ for $a$ large enough.
  Then $\int_{a_0}^\infty \nu'(a)da<\infty$, so the asymptotic density of the set $E_T=\{a\leq T|\nu'(a)\geq \frac{1}{a} \}$ satisfies
  $$\lim_{T\to \infty}\frac{\mu(E_T)}{T}=0,$$
  where $\mu$ is the Lebesgue measure.
  Since $\tilde t_a=t_{a+1}+O(\frac{\log a}{\sqrt{a}})$, we have $\lim_{a\to \infty}\nu'(a)=0$. 
  There are two possibilities:
  \begin{itemize}
    \item  $\nu'(a)$ is monotone, then we have $\nu'(a)=o(\frac{1}{a})$ and there is a sequence $a_j\to\infty$ so that $\nu''(a_j)=O(\frac{1}{a_j^2})$; 
    \item $\nu'(a)$ is not monotone, then around a local minimum where $|\nu'(a)|<\frac{1}{a}$, we have $\nu''(a)=0$.
  \end{itemize}
  In both cases, we can apply the argument for the non-monotonic case to get the claimed estimates in the theorem.
  The case $\nu'(a)< 0$ is similar.
  
\end{proof}
As a direct application, we get the following theorem.
\begin{theorem}\label{theo-log-sqrt}
  We have
  \begin{equation}
    \frac{h_a(x_a)}{\sqrt{x_a}}=\sqrt{2\pi}+O(\frac{\log a}{\sqrt{a}}),
    \end{equation}
    \begin{equation}
    x_a\lambda''(a)=1+O(\frac{\log a}{\sqrt{a}}),
    \end{equation}and 
    \begin{equation}
      \frac{1}{x}\frac{d^2}{dt^2}\log f(x)=1+O(\frac{\log x}{\sqrt{x}})
    \end{equation}
\end{theorem}
\begin{proof}
  The second and third estimates follow directly from theorem \ref{theo-general-log}. The first estimate follows from theorem \ref{theo-general-log} and the equation $$\frac{h_a(x_{a+1})}{x_{a+1}}=\int_{0}^{\infty}e^{-g_a(t)}dt.$$
\end{proof}
\begin{cor}\label{cor-tt-1}
  We have 
  $\tilde{t}_a-t_{a+1}=O(\frac{\log a}{a})$.
\end{cor}
\begin{proof}
  We have $\tilde{t}_a-t_{a+1}=\int_{0}^{\infty}\frac{(t-t_{a+1})dx}{h_a(x)}$. Write $\frac{dx}{h_a(x)}=\eta(a)e^{-g_a(t)}dt$.  Since $\int_{0}^{\infty}\frac{dx}{h_a(x)}=1$, 
  $$\tilde{t}_a-t_{a+1}=\frac{\int_{-\infty}^{\infty}(t-t_{a+1})e^{-g_a(t)}dt}{\int_{-\infty}^{\infty}e^{-g_a(t)}dt}.$$
  By theorem \ref{theo-log-sqrt} we have $$\int_{-\infty}^{\infty}e^{-g_a(t)}dt=(1+O(\frac{\log a}{\sqrt{a}}))\sqrt{\frac{2\pi}{x_{a+1}}} ,$$
  $$\int_{t_{a+1}}^{\infty}(t-t_{a+1})e^{-g_a(t)}dt=(1+O(\frac{\log a}{\sqrt{a}}))\frac{1}{x_{a+1}},$$
  and $$\int_{-\infty}^{t_{a+1}}(t-t_{a+1})e^{-g_a(t)}dt=(1+O(\frac{\log a}{\sqrt{a}}))\frac{-1}{x_{a+1}}.$$
  Therefore $$\tilde{t}_a-t_{a+1}=O(\frac{\log a}{\sqrt{a}})\cdot \frac{1}{\sqrt{x_{a+1}}}=O(\frac{\log a}{a}).$$
\end{proof}
In the same way, one can prove the following corollaries:
\begin{cor}
  $$\int\frac{(\log x-t_{a+1})^3  dx}{h_a(x)}=O(\frac{\log a}{a^{2}}).$$
\end{cor}
\begin{cor}
  We have 
  $a-n_{x_a}=O(\log a)$.
\end{cor}
\begin{proof}
  We can write $$a-n_{x_a}=\frac{\sum_{i=0}^{\infty}(i-n_{x_a})e^{-G_a(i)}}{\sum_{i=0}^{\infty}e^{-G_a(i)}}$$
  We have \begin{eqnarray*}
    \sum_{i=0}^{n_{x_a}}(i-n_{x_a})e^{-G_a(i)}&=&(1+O(\frac{1}{\sqrt{a}}))\int_{0}^{n_{x_a}}(i-n_{x_a})e^{-G_a(i)}di\\
    &=&(1+O(\frac{\log a}{\sqrt{a}}))\int_{0}^{n_{x_a}}(i-n_{x_a})e^{-\frac{1}{2x_a}(i-n_{x_a})^2}di\\
    &=&(1+O(\frac{\log a}{\sqrt{a}}))x_a
  \end{eqnarray*}
  And similarly, $$\sum_{i=n_{x_a}}^{\infty }(i-n_{x_a})e^{-G_a(i)}=-(1+O(\frac{\log a}{\sqrt{a}}))x_a.$$
   Therefore $$\sum_{i=0}^{\infty}(i-n_{x_a})e^{-G_a(i)}=O(\sqrt{a}\log a).$$
  Then since $$\sum_{i=0}^{\infty}e^{-G_a(i)}=(1+O(\frac{\log a}{\sqrt{a}}))\sqrt{2\pi x_a} ,$$
  we get the conclusion.
\end{proof}

\begin{cor}
  $$\frac{\sum_{i=0}^{\infty}(i-n_{x_a})^3e^{-G_a(i)}}{\sum_{i=0}^{\infty}e^{-G_a(i)}}=O(a\log a).$$
\end{cor}

To simplify the notations, we write $J_i=\int_0^\infty\frac{(\log x-\tilde{t}_a)^idx}{h_a(x)}$ and $H_j=\frac{\sum_{i=0}^{\infty}(i-a)^j g_a(i)}{\sum_{i=0}^{\infty} g_a(i)}$ , we then have the following:
$$\lambda''(a)=J_2,$$
$$\frac{d^3}{da^3}\lambda(a)=J_3, $$
$$\frac{d^2}{dt^2}\log f(x_a)=H_2,$$
and 
$$\frac{d^3}{dt^3}\log f(x_a)=H_3. $$

We also denote by $I_i=\int_0^\infty\frac{(\log x-t_{a+1})^idx}{h_a(x)}$, $\delta_a=t_{a+1}-\tilde{t}_a$, $K_j=\frac{\sum_{i=0}^{\infty}(i-n_{x_a})^j g_a(i)}{\sum_{i=0}^{\infty} g_a(i)}$ and $\sigma_a=a-n_{x_a}$. Then 
$$J_3=I_3+3\delta_aI_2-2\delta_a^3,$$
and $$H_3=K_3+3\sigma_aK_2-2\sigma_a^3.$$
Since
$$I_2=J_2+\delta_a^2=\frac{1}{x_{a+1}}+O(\frac{(\log a)^2}{a^{2}}),$$
and we have proved
$$I_3=O(\frac{\log a}{a^{2}}), $$
we have $$J_3=O(\frac{\log a}{a^{2}})$$
Therefore, we have proved the following
\begin{prop}\label{prop-lambda3}
  $$\frac{d^3}{da^3}\lambda(a)=O(\frac{\log a}{a^{2}}).$$
\end{prop}

Similarly, we have 
\begin{prop}\label{prop-f3}
  $$\frac{d^3}{dt^3}\log f(x)=O(x\log x).$$
\end{prop}

\section{Tightening up the loose ends}\label{sec-looseends}
One notice that in order to get better estimates than theorem \ref{theo-log-sqrt} we have to deal with two loose ends. The first is that on the interval $[a-\log a\sqrt{a},a+\log a\sqrt{a}]$, from $M_1<b\lambda''(b)<M_2$, we can only get $$M_1(1-1.1\frac{\log a}{\sqrt{a}})<a\lambda''(b)<M_2(1+1.1\frac{\log a}{\sqrt{a}}).$$
The second loose end lies in the error in estimating the summations over $a$ with integrals in estimating $\frac{d^2}{dt^2}\log f(x)$, etc. 
We deal with the second loose end first.
\subsection{the second loose end}

Now we show that the error in estimating the summations using integrals has a relative error smaller than $\frac{1}{\sqrt{a}}$.

Notice that we can extend the defintion of $c(i)=(\int_{0}^{\infty}\frac{x^i dx}{f(x)})^{-1}$ to $i>-1$ and it is easy to see that
 $$\lim_{i\to -1}c(i)=0.$$ Of course in this way, $c(0)\neq c_0$, but the contribution of this difference to our calculations will be negligible.

We define a function on $\R$
$$\zeta_a(i)=
\begin{cases}
\frac{c(i)}{c(a)}x_a^{i-a}& \textbf{if } i>-1\\
0& \textbf{otherwise}
\end{cases}
$$ 
Then it is clear that $\zeta_a$ is continuous, and as $i\to +\infty$ $\zeta_a$ has fast decay, hence integrable.

Let $$ \hat{g}(\xi)=\int_{-\infty}^{\infty}e^{-2i\pi \xi x}g(x)dx $$
denote the Fourier transform of a function $g(x)$. We recall the Poisson summation formula:
\begin{theorem}[\cite{Stein-FE}]
	Suppose $g(x)=\int_{\R} \hat{f}e^{2\pi ixy dy}$ with $|g(x)|\leq A(1+|x|)^{-1-\delta}$ and $|\hat{g}(y)|\leq A(1+|y|)^{-1-\delta}$, $\delta>0$. Then 
	$$\sum_{c=-\infty}^{\infty}g(x+c)=\sum_{\xi=-\infty}^{\infty}\hat{g}(\xi)e^{2\pi i\xi x}.$$
\end{theorem}
Clearly we have $|\zeta_a(x)|\leq A(1+|x|)^{-1-\delta}$.

We need to show that $|\hat{g}_a(y)|\leq A(1+|y|)^{-1-\delta}$ for some $\delta>0$. Since
$$\int_{0}^{1}x^i dx=\frac{1}{1+i}, \quad \int_{0}^{1}x^i\log x dx=\frac{1}{(1+i)^2},$$
we can estimate the derivatives as $i\to -1$:
\begin{eqnarray*}
  c'(i)&=&(\int_{0}^{\infty}\frac{x^i\log x dx}{f(x)})/(\int_{0}^{\infty}\frac{x^i dx}{f(x)})^2\\
  &=& \frac{(1+i)^{-2}+O(1)}{((1+i)^{-1}+O(1))^2}\\
  &=& 1+O((1+i)^{-2})
\end{eqnarray*}
Similarly, since $$\int_{0}^{1}x^i(\log x)^2 dx=\frac{2}{(1+i)^3},$$
\begin{eqnarray*}
  c''(i)&=&(\int_{0}^{\infty}\frac{x^i dx}{f(x)})^{-3}[\int_{0}^{\infty}\frac{x^i dx}{f(x)}\int_{0}^{\infty}\frac{x^i(\log x)^2 dx}{f(x)}-2(\int_{0}^{\infty}\frac{x^i \log x dx}{f(x)})^2 ]\\
  &=&\frac{((1+i)^{-1}+O(1))(2(1+i)^{-3}+O(1))-((1+i)^{-2}+O(1))^2}{((1+i)^{-1}+O(1))^{3}} \\
  &=& O(1),
\end{eqnarray*}
as $i\to -1$. For the third derivative, we need a little more effort. We have $$\int_{0}^{\infty}\frac{x^i dx}{f(x)}=\frac{1}{1+i}+w(i),$$
for some function $w(i)$ continuous in a neighborhood of $-1$. We denote by $T=\frac{1}{1+i}$. 
We have 
$$\int_{0}^{1}x^i(\log x)^3 dx=\frac{6}{(1+i)^4}.$$
So \begin{eqnarray*}
  c^{(3)}(i)&=&(\int_{0}^{\infty}\frac{x^i dx}{f(x)})^{-4}[(\int_{0}^{\infty}\frac{x^i dx}{f(x)})^2\int_{0}^{\infty}\frac{x^i(\log x)^3 dx}{f(x)}-6\int_{0}^{\infty}\frac{x^i dx}{f(x)}\cdot\\
  &&\int_{0}^{\infty}\frac{x^i \log x dx}{f(x)}\int_{0}^{\infty}\frac{x^i(\log x)^2 dx}{f(x)}+6(\int_{0}^{\infty}\frac{x^i \log x dx}{f(x)})^3 ]\\
  &=&\frac{1}{(T+O(1))^{4}}[(T+w(i))^2(6T^4+O(1))-6(T+w(i))\cdot\\
  &&(T^2+O(1))(2T^3+O(1))+6(T^2+O(1))^3] \\
  &=& O(1),
\end{eqnarray*}
Therefore, we have $\exists C>0$ independent of $a$, such that 
\begin{equation}\label{e-ga3}
  |\zeta_a^{(3)}(i)|\leq \frac{C(\log x_a)^3}{c(a)x_a^{a}} ,
\end{equation}
for $-1<i\leq 0$.
Then we can prove the following:

\begin{prop}\label{prop-ghat}
  For $|y|\geq 1$
  $$|\hat{\zeta}_a(y)|\leq \frac{\epsilon(a)}{y^2}+\frac{1}{|y|^3}O(\frac{1}{a}),$$
  where the $\epsilon(a)$ term and the $O(\frac{1}{a})$ term can be made independent of $y$.
\end{prop}
\begin{proof}
  By integration by parts, 
  \begin{eqnarray*}
    \int_{-\infty}^{\infty}e^{-2i\pi y z}\zeta_a(z)dz&=&\frac{i}{2\pi y}\int_{-1}^{\infty}e^{-2i\pi y z}\zeta'_a(z)dz\\
    &=&(\frac{i}{2\pi y})^2e^{2i\pi y}\zeta'_a(-1)-(\frac{i}{2\pi y})^2\int_{-1}^{\infty}e^{-2i\pi y z}\zeta''_a(z)dz,
  \end{eqnarray*}
  where by $\zeta'_a(-1)$ we mean $\lim_{i\to -1^+}\zeta'_a(i)$.
  By the previous calculations, we have 
  $$|(\frac{i}{2\pi y})^2e^{2i\pi y}\zeta'_a(-1)|\leq C_1\frac{x_a^{-1-a}}{y^2c_a},$$
  for some $C_1$ independent of $a$ and $y$. By integration by parts again
  $$ \int_{-1}^{\infty}e^{-2i\pi y z}\zeta''_a(z)dz=\frac{i}{2\pi y}e^{2i\pi y}\zeta''_a(-1)-\frac{i}{2\pi y}\int_{-1}^{\infty}e^{-2i\pi y z}\zeta^{(3)}_a(z)dz.$$
  We have $$|(\frac{i}{2\pi y})^3e^{2i\pi y}\zeta''_a(-1)|\leq C_2\frac{x_a^{-1-a}\log x_a}{|y|^3c_a},$$
  for some $C_2$ independent of $a$ and $y$. To estimate the integral 
  $\int_{-1}^{\infty}e^{-2i\pi y z}\zeta^{(3)}_a(z)dz$, we divide the integration interval into $[-1,\frac{a}{2}]$ and $[\frac{a}{2},\infty)$. For $0\leq i\leq \frac{a}{2}$, since 
  $$\int_{x_a}^{\infty}\frac{x^i dx}{f(x)}=\epsilon(a)\int_{0}^{\infty}\frac{x^i dx}{f(x)},$$
  and $$\int_{x_a}^{\infty}\frac{x^i(\log x)^3 dx}{f(x)}=\epsilon(a)\int_{0}^{\infty}\frac{x^i dx}{f(x)},$$ where both the $\epsilon(a)$ terms can be made independent of $i$, 
  we have $$|\frac{1}{c(i)}\frac{d^3}{di^3}c(i)|\leq (\log x_a)^3.$$
 Therefore $$|\frac{d^3}{di^3}c(i)|\leq C_3c(i)(\log x_a)^3,$$
  for some $C_3>0$ independent of $i$ and $a$. 
  So $$|\frac{d^3}{di^3}\zeta_a(i)|\leq C_4\zeta_a(i)(\log x_a)^3,$$
  for some $C_4>0$ independent of $i$ and $a$.
  Therefore $$\int_{-1}^{\frac{a}{2}}e^{-2i\pi y z}\zeta^{(3)}_a(z)dz=\epsilon(a).$$
Write $\zeta_a(i)=e^{\mu(i)}$. Then since $\zeta'_a(n_{x_a})=0$ and $\mu''(i)=-\lambda''(i)=O(\frac{1}{a})$ for $i>\frac{a}{2}$, we have 
$$\mu'(i)=O(\frac{1}{a})(i-n_{x_a}).$$ 
By proposition \ref{prop-lambda3}, we have $\mu^{(3)}(i)=O(\frac{\log a}{a^2})$ for $i>\frac{a}{2}$.
Therefore
\begin{eqnarray*}
  \int_{\frac{a}{2}}^{\infty}e^{-2i\pi y z}\zeta^{(3)}_a(z)dz&=&\int_{\frac{a}{2}}^{\infty}e^{-2i\pi y z}[\mu^{(3)}(z)+3\mu^{(2)}(z)\mu'(z)+(\mu'(z))^3]\zeta_a(z)dz\\
  &=&\int_{\frac{a}{2}}^{\infty}[O(\frac{\log a}{a^2})+O(\frac{1}{a^2})|z-n_{x_a}|+O(\frac{1}{a^3})|z-n_{x_a}|^3]\zeta_a(z)dz\\
  &=&O(\frac{1}{a})
\end{eqnarray*}
So we have proved the lemma.
\end{proof}
As a direct corollary, 
$|\hat{g}_a(y)|\leq A(1+|y|)^{-1-\delta}$ for some $\delta>0$. And since $\zeta_a$ is continuous, $\zeta_a(x)=\int_{-\infty}^{\infty}\hat{g}_a(y) e^{2\pi i xy dy}$. So we can use the Poisson summation formula to get the following:
\begin{lem}\label{lem-sum-integral}
  We have
  $$\sum_{i=0}^\infty \zeta_a(i)=(1+O(\frac{1}{a^{3/2}}))\int_{i=0}^\infty \zeta_a(i)di, $$
  for $a$ large enough, 
\end{lem}
\begin{proof}
Since $$\int_{i=0}^\infty \zeta_a(i)di=\sqrt{2\pi x_{a+1}}(1+O(\frac{1}{\sqrt{a}})),$$
  the lemma follows directly from the Poisson summation formula and proposition \ref{prop-ghat}.

\end{proof}
\begin{rem}
  \begin{itemize}
    \item In the same way, it is not hard to prove: $$\sum_{i=0}^\infty (i-n_{x_a})^j \zeta_a(i)=(1+O(\frac{1}{a^{3/2}}))\int_{i=0}^\infty (i-n_{x_a})^j \zeta_a(i)di, $$ for $a$ large enough, for $j=2$.
    \item When $j$ is odd, it is harder to get a relative error than the absolute error instead. When $j=1$, we have
     $$\sum_{i=0}^\infty (i-n_{x_a}) \zeta_a(i)=\int_{i=0}^\infty (i-n_{x_a}) \zeta_a(i)di+O(\frac{1}{a^{1/2}}) $$ 
     And when $j=3$, we have
      $$\sum_{i=0}^\infty (i-n_{x_a})^3 \zeta_a(i)=\int_{i=0}^\infty (i-n_{x_a})^3 \zeta_a(i)di+O(a^{1/2}) $$ 
  \end{itemize}
  
\end{rem}
\subsection{the first loose end}\label{subsec-fls}
We define two convex functions $\phi_a(t)$ and $\psi_a(t)$ on the interval $P_a=[x_{a+1}-\sqrt{a}\log a,x_{a+1}+\sqrt{a}\log a]$ that satisfies the following conditions 
$$\begin{cases}
  \phi_a(t_{a+1})=\phi'_a(t_{a+1})=0;\\
  \phi_a''(t)=e^t,
\end{cases}$$
and $$\begin{cases}
  \psi_a(t_{a+1})=\psi'_a(t_{a+1})=0;\\
  \psi_a''(t)=x_{a+1}.
\end{cases}$$
Then let $\Delta g=g_a-\phi_a$, we then have $$\Delta g(t)=O(\sqrt{a}\log a (t-t_{a+1})^2)=O(\frac{(\log a)^3}{\sqrt{a}}).$$
Therefore \begin{eqnarray*}
  e^{-g(t)}&=&e^{-\phi_a-\Delta g}\\
  &=&e^{-\phi_a}(1-\Delta g+O((\Delta g)^2) )
\end{eqnarray*}
We then let $\tilde{g}=\psi_a+\Delta g$ and get $$e^{-\tilde{g}(t)}=e^{-\psi_a}(1-\Delta g+O((\Delta g)^2) ).$$
Then we have
\begin{eqnarray*}
  \int (t-\bar{t}(g_a))^2e^{-g_a(t)}dt&=& (1+O(\frac{(\log a)^2}{a}))\int \tau^2 e^{-g_a(t)}dt\\
  &=& (1+O(\frac{(\log a)^2}{a}))\int \tau^2 e^{-\phi_a(t)}(1-\Delta g+O((\Delta g)^2) )dt\\
  &=& (1+O(\frac{(\log a)^6}{a}))\int \tau^2 e^{-\phi_a(t)}(1-\Delta g)dt.
\end{eqnarray*}
Similarly
\begin{eqnarray*}
  \int (t-\bar{t}(\tilde{g}_a))^2e^{-\tilde{g}_a}dt&=& (1+O(\frac{(\log a)^2}{a}))\int \tau^2 e^{-\tilde{g}_a}dt\\
  &=& (1+O(\frac{(\log a)^2}{a}))\int \tau^2 e^{-\psi_a(t)}(1-\Delta g+O((\Delta g)^2) )dt\\
  &=& (1+O(\frac{(\log a)^6}{a}))\int \tau^2 e^{-\psi_a(t)}(1-\Delta g)dt.
\end{eqnarray*}
Since $$\int e^{-\phi_a}d\tau=(1+O(\frac{1}{a}))\int e^{-\psi_a}d\tau, $$
$$\int \tau^2 e^{-\phi_a}d\tau=(1+O(\frac{1}{a}))\int \tau^2e^{-\psi_a}d\tau, $$
and \begin{eqnarray*}
  \int \tau^2 e^{-\phi_a(t)}\Delta_g-\int \tau^2 e^{-\psi_a(t)}\Delta_g&=&\int \tau^2 (e^{-\phi_a(t)}-e^{-\psi_a(t)})\Delta_g\\
  &=&\int \tau^2e^{-\frac{\tau^2}{2x_{a+1}}}O(a\tau^3)O(\sqrt{a}\log a \tau^2)d\tau\\
  &=&O(\frac{\log a}{a^{5/2}}),
\end{eqnarray*}
we get 
\begin{equation}
  \int (t-\bar{t}(g_a))^2e^{-g_a(t)}dt=(1+O(\frac{(\log a)^6}{a}))\int_{P_a} (t-\bar{t}(\tilde{g}_a))^2e^{-\tilde{g}_a}dt.
\end{equation}
And similarly \begin{equation}\label{e-g-0}
  \int e^{-g_a(t)}dt=(1+O(\frac{(\log a)^6}{a}))\int_{P_a} e^{-\tilde{g}_a}dt.
\end{equation}
Therefore we get \begin{equation}\label{e-g-2}
  \tilde{d}(g_a)=(1+O(\frac{(\log a)^6}{a})) \tilde{d}(\tilde{g}_a).
\end{equation}
And the following proposition.
\begin{prop}\label{prop-lambda2-2}
  We have $$\lambda''(a)=(1+O(\frac{(\log a)^6}{a}))\gamma(a),$$
  where $$\gamma(a)=\frac{\int_{P_a} (t-\bar{t}(\tilde{g}_a))^2e^{-\tilde{g}_a}dt}{\int_{P_a} e^{-\tilde{g}_a}dt}.$$
\end{prop}
And the advantage of replacing $g_a$ with $\tilde{g}_a$ is that when $$M_1\leq \frac{g_a''(x)}{x}\leq M_2,$$ with $M_1=1-O(\frac{\log a}{\sqrt{a}})$ and $M_2=1+O(\frac{\log a}{\sqrt{a}})$, for $x\in [x_{a+1}-\sqrt{a}\log a,x_{a+1}+\sqrt{a}\log a] $, we have 
$$(M_1-1)(x_{a+1}+\sqrt{a}\log a)\leq (\Delta g)''(x)\leq (M_2-1)(x_{a+1}+\sqrt{a}\log a).$$
Therefore $$x_{a+1}+(M_1-1)(x_{a+1}+\sqrt{a}\log a)\leq \tilde{g}''(x)\leq x_{a+1}+(M_2-1)(x_{a+1}+\sqrt{a}\log a).$$
And the ratio of the upper and the lower bounds is 
\begin{eqnarray*}
  \frac{x_{a+1}+(M_2-1)(x_{a+1}+\sqrt{a}\log a)}{x_{a+1}+(M_1-1)(x_{a+1}+\sqrt{a}\log a)}&=&\frac{1+(M_2-1)(1+\frac{\log a \sqrt{a}}{x_{a+1}})}{1+(M_1-1)(1+\frac{\log a \sqrt{a}}{x_{a+1}})}\\
  &=&\frac{M_2}{M_1}(1+O(\frac{(\log a)^2}{a})),
\end{eqnarray*}
namely $\exists C_0>0$ so that 
\begin{equation}\label{e-use-et}
  \frac{x_{a+1}+(M_2-1)(x_{a+1}+\sqrt{a}\log a)}{x_{a+1}+(M_1-1)(x_{a+1}+\sqrt{a}\log a)}<\frac{M_2}{M_1}(1+C_0\frac{(\log a)^2}{a}).
\end{equation}

For the integrals with respect to $e^{-G_a(i)}di$ on the interval $Q_a=[a-\sqrt{a}\log a,a+\sqrt{a}\log a]$, we can do similar things. We define  
two convex functions $\Phi_a(t)$ and $\Psi_a(t)$ that satisfies the following conditions 
$$\begin{cases}
  \Phi_a(n_{x_a})=\phi'_a(n_{x_a})=0;\\
  \Phi_a''(i)=\frac{1}{x_i},
\end{cases}$$
and $$\begin{cases}
  \Psi_a(n_{x_a})=\psi'_a(n_{x_a})=0;\\
  \Psi_a''(i)=\frac{1}{x_{a}}.
\end{cases}$$
Then let $\Delta G=G_a-\Phi_a$, we have $$\Delta G=O(\frac{\log a}{a^{3/2}}j^2),$$where $j=i-n_{x_a}$, and 
$$\frac{M_1-1}{x_a-2\sqrt{a}\log a}\leq (\Delta G)''(i)\leq \frac{M_2-1}{x_a-2\sqrt{a}\log a},$$
for $i\in Q_a$.

We need to estimate $\int e^{-\Phi_a(i)}di$ and $\int j^2e^{-\Phi_a(i)}di$ .
$$\frac{da}{dx_a}=\frac{1}{x_a}\frac{d^2}{dt^2}\log f(x_a),$$
So $$\frac{dx_a}{da}=\frac{x_a}{\frac{d^2}{dt^2}\log f(x_a)}=1+O(\frac{\log a}{\sqrt{a}}), $$
and 
$$\frac{dx^{-1}_a}{da}=-\frac{\frac{dx_a}{da}}{x_a^2}=-\frac{1+O(\frac{\log a}{\sqrt{a}})}{x_a^2}.$$
Moreover, \begin{eqnarray*}
  \frac{d^2x_a}{da^2}&=& \frac{dx_a}{da}(\frac{d^2}{dt^2}\log f(x_a))^{-1}-(\frac{d^2}{dt^2}\log f(x_a))^{-2}x_a \frac{d^3}{dt^3}\log f(x_a) \frac{dt_a}{da}\\
  &=&O(\frac{\log a}{a})
\end{eqnarray*}
Therefore \begin{eqnarray*}
  \frac{d^2x^{-1}_a}{da^2}&=&-\frac{\frac{d^2x_a}{da^2}}{x_a^2}+2\frac{(\frac{dx_a}{da})^2}{x_a^3}\\
  &=&O(\frac{\log a}{a^3})
\end{eqnarray*}
We denote by $\tilde{a}=n_{x_a}$, $D_3=\frac{dx^{-1}_a}{da}|_{a=\tilde{a}}$, then
$$\Phi_a(i)=\frac{1}{2x_{\tilde{a}}}j^2+\frac{D_3}{6}j^3+O(\frac{\log a}{a^3}j^4), $$and 
\begin{equation}\label{e-Phi-a}
  e^{-\Phi_a(i)}=e^{-\frac{1}{2x_{\tilde{a}}}j^2}[1-\frac{D_3}{6}j^3+O(\frac{(\log a)^8}{a})] 
\end{equation}

 So \begin{eqnarray*}
  \int_{Q_a} e^{-\Phi_a(i)}di&=& \int_{Q_a}e^{-\frac{1}{2x_{\tilde{a}}}j^2}[1-\frac{D_3}{6}j^3+O(\frac{(\log a)^8}{a})]di\\
  &=&(1+(\frac{(\log a)^8}{a}))\sqrt{2\pi x_{\tilde{a}}}\\
  &=&(1+(\frac{(\log a)^8}{a}))\sqrt{2\pi x_a},
 \end{eqnarray*}
where the last equation is because $x_{\tilde{a}}-x_a=O(\log a)$.
Similarly $$\int_{Q_a} j^2e^{-\Phi_a(i)}di=(1+O(\frac{(\log a)^8}{a}))x_a \sqrt{2\pi x_a}. $$
On the other hand, it is easy to get that 
$$\int_{Q_a} e^{-\Psi_a(i)}di=(1+\epsilon(a))\sqrt{2\pi x_a}, $$
and $$\int_{Q_a} j^2e^{-\Psi_a(i)}di=(1+\epsilon(a))x_a\sqrt{2\pi x_a}. $$

Since $a-n_{x_a}=O(\log a)$, we also have \begin{eqnarray*}
  \int_{Q_a} (e^{-\Psi_a(i)}-e^{-\Phi_a(i)})\Delta G di&=&\int_{Q_a} e^{-\frac{1}{2x_{\tilde{a}}}j^2}O(\frac{\log a}{a^2} j^2+\frac{1}{a^2} j^3)O(\frac{\log a}{a^{3/2}} j^2)\\
  &=&O(\frac{\log a}{\sqrt{a}}),
\end{eqnarray*}
and similarly $$\int_{Q_a} j^2(e^{-\Psi_a(i)}-e^{-\Phi_a(i)})\Delta G di=O(\log a\sqrt{a}).$$
Therefore, let $\tilde{G}_a(i)=G_a(i)-\Phi_a(i)+\Psi_a(i)$, we have
\begin{equation}\label{e-G-0}
  \int_{Q_a} e^{-G_a(i)}di=(1+O(\frac{(\log a)^8}{a}))\int_{Q_a} e^{-\tilde{G}_a(i)}di, 
\end{equation}
and 
\begin{equation}\label{e-G-2}
  \int_{Q_a} (i-\bar{t}(G_a))^2e^{-G_a(i)}di=(1+O(\frac{(\log a)^8}{a}))\int_{Q_a}(i-\bar{t}(\tilde{G}_a))^2 e^{-\tilde{G}_a(i)}di.
\end{equation}
So we have the following proposition.
\begin{prop}\label{prop-logf2-2}
  We have $$\frac{d^2}{dt^2}\log f(x_a)=(1+O(\frac{(\log a)^8}{a}))\Gamma(a),$$
  where $$\Gamma(a)=\frac{\int_{Q_a}(i-\bar{t}(\tilde{G}_a))^2 e^{-\tilde{G}_a(i)}di}{\int_{Q_a} e^{-\tilde{G}_a(i)}di}.$$
\end{prop}
We also have 
$$x^{-1}_{a}+ \frac{M_1-1}{x_a-2\sqrt{a}\log a}\leq \tilde{G}_a''\leq x^{-1}_{a}+ \frac{M_2-1}{x_a-2\sqrt{a}\log a},$$
and that $\exists C_1>0$ so that the ratio of the upper and lower bounds satisfies
\begin{equation}\label{e-use-xi}
  \frac{x^{-1}_{a}+(M_2-1)(x_a-2\sqrt{a}\log a)^{-1}}{x^{-1}_{a}+(M_1-1)(x_a-2\sqrt{a}\log a)^{-1}}<\frac{M_2}{M_1}(1+C_1\frac{(\log a)^2}{a}).
\end{equation}
\section{Thinner gap}\label{sec-5}
So now we can go over again the arguments in section \ref{sec-3} and get finer estimates as follows.
Let $C_2=\max\{C_0,C_1\}$, where $C_0,C_1$ are constants from formulas \ref{e-use-et} and \ref{e-use-xi}.
We now let $m(a)=\frac{M_2(a)}{M_1(a)}(1+C_2\frac{(\log a)^2}{a})$.

\subsection*{non-monotonic case} 

At a local minimum point $b$, we get 
$$\lambda''(b)-\frac{1}{\frac{d^2}{dt^2}\log f(x_b)}+\frac{\frac{d^3}{dt^3}\log f(x_b)}{2(\frac{d^2}{dt^2}\log f(x_b))^3}\geq 0,$$
which, by proposition \ref{prop-f3}, implies that $$\lambda''(a)\frac{d^2}{dt^2}\log f(x_b)\geq 1-O(\frac{\log b}{b}).$$
Then by propositions \ref{prop-lambda2-2} and \ref{prop-logf2-2}, we have 
$$\gamma(a)\Gamma(a)\geq 1+O(\frac{(\log a)^8}{a}).$$

let $p_1=\tilde{p}(m(b))$, $q_1=\tilde{q}(m(b))$. Since 
$$h_b(x_b)=(1+O(\frac{1}{b}))\int_{0}^{\infty}e^{-G_b(i)}di=(1+O(\frac{(\log b)^8}{b}))\int_{Q_b}e^{-\tilde{G}_b(i)}di,$$
and $$h_b(x_b)=x_b(1+O(\frac{1}{a}))\int_{0}^{\infty}e^{-g_b(i)}di=(1+O(\frac{(\log b)^8}{b}))x_b\int_{P_b}e^{-\tilde{g}_b(t)}dt,$$
we get 
$$\Gamma(b)\leq p_1(1+O(\frac{(\log b)^8}{b}))h^2_b(x_b),$$ and
$$\gamma(b)\leq p_1(1+O(\frac{(\log b)^8}{b}))\frac{h^2_b(x_b)}{x_b^2}.$$

So we get that 
$(\nu(b))^4p_1^2\geq 1+O(\frac{(\log b)^8}{b})$. So $$\nu(b)>(1-C_3\frac{(\log b)^8}{b})\frac{1}{\sqrt{p_1}}$$ for $b$ large enough, for some $C_3>0$ independent of $b$.
And similarly, let $b'>b$ be a local maximum, then we have $$\nu(b')<(1+C_4\frac{(\log b)^8}{b})\frac{1}{\sqrt{q_1}},$$for some $C_2>0$ independent of $b$. 
Let $C_5=\max \{C_3,C_4 \}$. Then by the same argument as in section \ref{sec-3}, we have
$$ (1-C_5\frac{(\log b)^8}{b})\frac{1}{\sqrt{p_1}}< \nu(a)< (1+C_5\frac{(\log b)^8}{b})\frac{1}{\sqrt{q_1}},$$
for all $a\geq b$. 

 Then we get a finer version of lemma \ref{lem-p1q1} by repeating the proof there.
\begin{lem}\label{lem-p1q12}
  There exists $C>0$ such that for $a$ large enough, assume $m(a)>1+C\frac{(\log a)^8}{a}$, then
  $$(1-C_5\frac{(\log a)^8}{a})\frac{1}{\sqrt{p_1}}< \nu(a')< (1+C_5\frac{(\log a)^8}{a})\frac{1}{\sqrt{q_1}},$$
  where $p_1=\tilde{p}(m(a))$ and $q_1=\tilde{q}(m(a))$,
  for all $a'\geq a+2\sqrt{a}\log a$.
\end{lem}
\begin{proof}
  We omit the first few lines in the proof of lemma \ref{lem-p1q1} and go directly to the case when $\mu$ is decreasing on $[a,a+2\sqrt{a}\log a]$ and that $\nu(a+2\sqrt{a}\log a)\geq (1+C_5\frac{(\log a)^8}{a})\frac{1}{\sqrt{q_1}}$, hence $$\nu(a')\geq (1+C_5\frac{(\log a)^8}{a})\frac{1}{\sqrt{q_1}},$$ for $a'\in [a,a+2\sqrt{a}\log a]$.

  Notice that for all $b\geq a$ we have
  \begin{eqnarray*}
    \gamma(b)&\geq& q_1(\int_{P_b}e^{-\tilde{g}_b(t)}dt+\epsilon(b))^2  \\
    &\geq&q_1\frac{\nu^2(b)}{x_b}(1+O(\frac{(\log a)^{8}}{a})).
  \end{eqnarray*}
  
  Therefore 
  $$x_b\gamma(b)\geq 1+C_6\frac{(\log a)^{8}}{a},$$
  for $b\in [a,a+2\sqrt{a}\log a]$, for some $C_6$ independent of $a$.
  So we have $$x_b\lambda''(b)\geq 1+C_7\frac{(\log a)^{8}}{a},$$
  for $b\in [a,a+2\sqrt{a}\log a]$, for some $C_7$ independent of $a$.
  Let $a_2=a+\sqrt{a}\log a$, then 
 \begin{eqnarray*}
  h_{a_2}(x_{a_2})&\leq& (1+C_8\frac{(\log a)^{8}}{a})\int_{Q_a}e^{-\Phi_{a_2}(i)}di\\
  &\leq & (1+C_9\frac{(\log a)^{8}}{a})\sqrt{2\pi x_{a_2}},
 \end{eqnarray*}
 for some $C_8, C_9$ independent of $a$. 
  So we have 
  $$\nu(a_2)\leq (1+C_9\frac{(\log a)^{8}}{a})\sqrt{2\pi},$$
which implies 
 \begin{equation}\label{e-impossible-0}
  (1+C_5\frac{(\log a)^{8}}{a})\frac{1}{\sqrt{q_1}}\leq (1+C_9\frac{(\log a)^{8}}{a})\sqrt{2\pi}.
 \end{equation}

Then one can see that there exists $C>0$ such that if $$m(a)>1+C\frac{(\log a)^{8}}{a} ,$$ then formula \ref{e-impossible-0} is impossible.

The case when $\mu$ is increasing is similar.
\end{proof}

Then we get a finer version of theorem \ref{theo-non-monotone} by repeating the proof there.
\begin{theorem}\label{theo-non-monotone-2}
  There exists $C>0$ such that for $a$ large enough, we have $$m(a)\leq 1+C\frac{(\log a)^8}{a}.$$
\end{theorem}

\subsection*{the general case} 
Furthermore, we get a finer version of theorem \ref{theo-general-log} by repeating the proof there.
\begin{theorem}\label{theo-general-a-log}
  For $a$ large engough, we have $$m(a)\leq 1+1000C_5\frac{(\log a)^8}{a}.$$
\end{theorem}
As a direct corollary, we get 
\begin{theorem}\label{theo-log9-a}
  We have 
  \begin{equation}
    \frac{h_a(x_a)}{\sqrt{x_a}}=\sqrt{2\pi}+O(\frac{(\log a)^8}{a}),
    \end{equation}
    
    \begin{equation}
    x_a\lambda''(a)=1+O(\frac{(\log a)^8}{a}),
    \end{equation}and 
    \begin{equation}
      \frac{1}{x}\frac{d^2}{dt^2}\log f(x)=1+O(\frac{(\log x)^8}{x}),
    \end{equation}
\end{theorem}

\subsection*{Applications.}
Then we can get a sharper version of corollary \ref{cor-tt-1}.
\begin{cor}\label{cor-tt-2}
  \begin{equation}
    \tilde{t}_a-t_{a+1}=-\frac{1}{2x_{a+1}}+O(\frac{(\log a)^8}{x_a^{3/2}}),
    \end{equation}
\end{cor}
\begin{proof}
  With the notations in subsection \ref{subsec-fls}, now we have 
  $$\Delta g(t)=O((\log a)^8(t-t_{a_1})^2)=O(\frac{(\log a)^{10}}{a}).$$
  Therefore \begin{eqnarray*}
    \int_{-\infty}^{\infty}e^{-g_a(t)}dt&=& \int_{-\infty}^{\infty}e^{-\phi_a(t)}(1-\Delta g(t)+O(\frac{(\log a)^{20}}{a^2}) )dt\\
    &=&(1+O(\frac{(\log a)^{8}}{a}))\int_{-\infty}^{\infty}e^{-\phi_a(t)}dt\\
    &=&(1+O(\frac{(\log a)^{8}}{a}))\sqrt{\frac{2\pi}{x_{a+1}} }
  \end{eqnarray*}
  Since $\phi_a(t)$ has Tayler expansion at $t_{a+1}$:
  $$\phi_a(t)=\frac{x_{a+1}}{2}\tau^2+\frac{x_{a+1}}{6}\tau^3+\frac{x_{a+1}}{24}\tau^4+\cdots, $$ 
  and that for $n\geq 0$ an integer, we have
  $$\int_{-\infty}^{\infty}e^{-x^2/2}x^{2n}dx=(2n-1)!!,$$
  we can calculate 
  $$\int_{Q_a}\tau e^{-\phi_a(t)}dt=(1+O(\frac{1}{a}))\frac{1}{2}\sqrt{\frac{2\pi}{x_{a+1}^3}}.$$
 Therefore, \begin{eqnarray*}
    \int_{-\infty}^{\infty}\tau e^{-g_a(t)}dt&=& \int_{-\infty}^{\infty}\tau e^{-\phi_a(t)}(1-\Delta g(t)+O(\frac{(\log a)^{20}}{a^2}) )dt\\
    &=&\int_{Q_a}\tau e^{-\phi_a(t)}dt+O(\frac{(\log a)^{8}}{a^{3/2}})\\
    &=&\frac{1}{2}\sqrt{\frac{2\pi}{x_{a+1}^3}} +O(\frac{(\log a)^{8}}{a^{2}})
  \end{eqnarray*}
  Therefore, we have $$\tilde{t}_a-t_{a+1}=-\frac{1}{2x_{a+1}}+O(\frac{(\log a)^8}{x_a^{3/2}}).$$
\end{proof}
Recall that we have defined $I_i=\int\frac{(\log x-t_{a+1})^idx}{h_a(x)}$. Similarly, we can prove the following corollary.
\begin{prop}\label{prop-j345}
  We have
  \begin{equation}\label{e-j3}
    \int\frac{(\log x-t_{a+1})^3  dx}{h_a(x)}=-\frac{5}{2x^2_{a+1}}+O(\frac{(\log a)^8}{a^{5/2}}),
  \end{equation}
  \begin{equation}\label{e-j4}
    I_4=\frac{3}{x^2_{a+1}}+O(\frac{(\log a)^8}{a^{3}}),
  \end{equation}
and
  \begin{equation}\label{e-j5}
    I_5=-\frac{35}{2x^3_{a+1}}+O(\frac{(\log a)^8}{a^{7/2}}).
  \end{equation}
\end{prop}
\begin{proof}
 We just need to calculate as in the proof of corollary \ref{cor-tt-2} the following:
 $$ \int_{Q_a}\tau^3 e^{-\phi_a(t)}d\tau= \sqrt{\frac{2\pi}{x_{a+1}}}[-\frac{5}{2x^2_{a+1}}+O(\frac{1}{a^3}) ],$$
 $$ \int_{Q_a}\tau^4 e^{-\phi_a(t)}d\tau= \sqrt{\frac{2\pi}{x_{a+1}}}[\frac{3}{x^2_{a+1}}+O(\frac{1}{a^3}) ],$$
  and 
 $$ \int_{Q_a}\tau^5 e^{-\phi_a(t)}d\tau= \sqrt{\frac{2\pi}{x_{a+1}}}[-\frac{35}{2x^3_{a+1}}+O(\frac{1}{a^3}) ].$$
 Notice that the reason that the relative error of $I_4$ is smaller is that 
 \begin{eqnarray*}
  \int_{-\infty}^{\infty}\tau^4 e^{-g_a(t)}dt&=& \int_{-\infty}^{\infty}\tau e^{-\phi_a(t)}(1-\Delta g(t)+O(\frac{(\log a)^{20}}{a^2}) )dt\\
  &=&(1+O(\frac{(\log a)^{8}}{a}))\int_{Q_a}\tau e^{-\phi_a(t)}dt.
\end{eqnarray*}
\end{proof}
Recall that $J_3=I_3+3\delta_aI_2-2\delta_a^3.$
Then since $I_2=\frac{1}{x_{a+1}}+O(\frac{(\log a)^8}{a^2})$, we have 
$$J_3=-\frac{1}{x^2_{a+1}}+O(\frac{(\log a)^8}{a^{5/2}}).$$
Namely, we have 
\begin{prop}\label{prop-lambda32}
  $$\frac{d^3}{da^3}\lambda(a)=-\frac{1}{x^2_{a+1}}+O(\frac{(\log a)^8}{a^{5/2}}).$$
\end{prop}
Since
\begin{eqnarray*}
  J_4&=&I_4+4\delta_a I_3+6\delta_a^2I_2-3\delta_a^4\\
  &=&\frac{3}{x^2_{a+1}}+O(\frac{(\log a)^8}{a^{3}}),
\end{eqnarray*}
we have 
$$J_4-3J_2^2=O(\frac{(\log a)^8}{a^{3}}),$$
namely,
\begin{prop}\label{prop-lambda4}
  $$\frac{d^4}{da^4}\lambda(a)=O(\frac{(\log a)^8}{a^{3}}).$$
\end{prop}
And since
\begin{eqnarray*}
  J_5&=&I_5+5\delta_a I_4+10\delta_a^2I_3+10\delta^3I_2-4\delta_a^5\\
  &=&-\frac{35}{2x^3_{a+1}}+5(\frac{1}{2x_{a+1}}+O(\frac{(\log a)^8}{x_a^{3/2}}))(\frac{3}{x^2_{a+1}}+O(\frac{(\log a)^8}{a^{3}})) +O(\frac{(\log a)^8}{a^{7/2}})\\
  &=&-\frac{10}{x^3_{a+1}}+O(\frac{(\log a)^8}{a^{7/2}})
\end{eqnarray*}
Therefore,
$$J_5-10J_2J_3=O(\frac{(\log a)^8}{a^{7/2}}),$$
namely,
\begin{prop}\label{prop-lambda5}
  $$\frac{d^5}{da^5}\lambda(a)=O(\frac{(\log a)^8}{a^{7/2}}).$$
\end{prop}

\subsection{tighening the second loose end for odd degrees.}
For $j=1,3,5$, we define a function on $\R$
$$\zeta_{a,j}(i)=
\begin{cases}
\frac{c(i)}{c(a)}x_a^{i-a}(i-n_{x_a})^j& \textbf{if } i>-1\\
0& \textbf{otherwise}
\end{cases}
$$ 
Then let $\hat{\zeta}_{a,j}(y)$ be the Fourier transform of $\zeta_{a,j}$. Then we can extend the proof of proposition \ref{prop-ghat} to prove the following:
\begin{prop}\label{prop-gjhat}
  For $j=1,3,5$ and $|y|\geq 1$, we have
  $$|\hat{\zeta}_{a,j}(y)|\leq \frac{\epsilon(a)}{y^2}+\frac{1}{|y|^3}O(a^{-2+j/2}),$$
  where the $\epsilon(a)$ term and the $O(a^{-2+j/2})$ term can be made independent of $y$.
\end{prop}
\begin{proof}
  We can repeat the argument in the proof of proposition \ref{prop-ghat} and find that it boils down to estimating the integral
  $$ K_j=\int_{\frac{a}{2}}^{\infty}e^{-2i\pi y z}\zeta^{(3)}_{a,j}(z)dz.$$
  We can then use integration by parts twice:
  \begin{eqnarray*}
    K_j&=& -\frac{i}{2\pi y}\zeta^{(3)}_{a,j}(\frac{a}{2})-\frac{i}{2\pi y}\int_{\frac{a}{2}}^{\infty}e^{-2i\pi y z}\zeta^{(4)}_{a,j}(z)dz\\
    &=&\frac{\epsilon(a)}{y} +(\frac{i}{2\pi y})^2\zeta^{(4)}_{a,j}(\frac{a}{2})+(\frac{i}{2\pi y})^2\int_{\frac{a}{2}}^{\infty}e^{-2i\pi y z}\zeta^{(5)}_{a,j}(z)dz\\
    &=&\frac{\epsilon(a)}{y}+(\frac{i}{2\pi y})^2\int_{\frac{a}{2}}^{\infty}e^{-2i\pi y z}\zeta^{(5)}_{a,j}(z)dz,
  \end{eqnarray*}
  where the $\epsilon(a)$-term can be controlled independent of $y$ and $j$.

  We denote by $\zeta_a(i)=e^{\mu(i)}$, $\theta =i-n_{x_a}$. When $j=1$, 
  $$ \zeta^{(5)}_{a,1}(z)=\theta \zeta^{(5)}_{a}+\zeta^{(4)}_{a},$$
  where 
  \begin{eqnarray*}
    \zeta^{(4)}_{a}=\mu^{(4)}+7\mu^{(3)}\mu'+3(\mu^{(2)})^2+6\mu^{(2)}(\mu')^2+(\mu')^4,
  \end{eqnarray*}
  and $$\zeta^{(5)}_{a}=\mu^{(5)}+8\mu^{(4)}\mu'+13\mu^{(3)}\mu^{(2)}+13\mu^{(3)}(\mu')^2+15(\mu^{(2)})^2\mu'+10\mu^{(2)}(\mu')^3
  +(\mu')^5$$
  Then since $\mu=-\lambda$, we have $\mu''(i)=-\lambda''(i)=O(\frac{1}{a})$, hence  
  $\mu'(i)=O(\frac{1}{a})\theta.$ And by propositions \ref{prop-lambda32}, \ref{prop-lambda4} and \ref{prop-lambda5}, we have $\mu^{(3)}=O(\frac{1}{a^2})$, $\mu^{(4)}=O(\frac{(\log a)^8}{a^{3}})$ and $\mu^{(5)}=O(\frac{(\log a)^8}{a^{7/2}})$. Then after a careful term by term estimation one sees that 
  $$\int_{\frac{a}{2}}^{\infty}e^{-2i\pi y z}\zeta^{(5)}_{a,1}(z)=O(\frac{1}{a^{3/2}}).$$
Similarly,  since $$ \zeta^{(5)}_{a,3}(z)=\theta^3 \zeta^{(5)}_{a}+3\theta^2 \zeta^{(4)}_{a}+6\theta \zeta^{(3)}_{a}+6\zeta^{(2)}_{a}, $$
we can again get estimate:
$$\int_{\frac{a}{2}}^{\infty}e^{-2i\pi y z}\zeta^{(5)}_{a,3}(z)=O(\frac{1}{a^{1/2}}).$$
 And since 
 $$ \zeta^{(5)}_{a,5}(z)=\theta^5 \zeta^{(5)}_{a}+5\theta^4 \zeta^{(4)}_{a}+20\theta^3 \zeta^{(3)}_{a}+60\theta^2\zeta^{(2)}_{a}+120\theta \zeta'_a+120g_a, $$
 we can again get an estimate:
 $$\int_{\frac{a}{2}}^{\infty}e^{-2i\pi y z}\zeta^{(5)}_{a,3}(z)=O(a^{1/2}).$$

\end{proof}

So by the Poisson summation formula, we have the following:
\begin{lem}\label{lem-sum-integral-odd}
  When $j=1,3,5$, we have
  $$\sum_{i=0}^\infty (i-n_{x_a})^j\zeta_a(i)=\int_{i=0}^\infty (i-n_{x_a})^j\zeta_a(i)+O(a^{-2+j/2}), $$
  for $a$ large enough, 
\end{lem}
As a direct application, we can prove:
\begin{prop}\label{prop-a2}
  We have 
  $$a-n_{x_a}=\frac12+O(\frac{(\log a)^8}{\sqrt{a}}).$$
\end{prop}
\begin{proof}
 Since now we can use integrals to estimate the summations, we can mimic the proof of corollary \ref{cor-tt-2}. Recall that 
 $$a-n_{x_a}=\frac{\sum_{i=0}^{\infty }(i-n_{x_a})e^{-G_a(i)}}{\sum_{i=0}^{\infty }e^{-G_a(i)}}.$$
  Recall that we have defined $\Phi_a(i)$ on the interval $Q_a=[a-\sqrt{a}\log a,a+\sqrt{a}\log a]$ in subsection \ref{subsec-fls}.
 Then $\Delta G(i)=O(\frac{(\log a)^8}{a^2})(i-n_{x_a})^2$. Formula \ref{e-Phi-a} says 
 $$e^{-\Phi_a(i)}=e^{-\frac{1}{2x_{\tilde{a}}}j^2}[1-\frac{D_3}{6}j^3+O(\frac{(\log a)^8}{a})], $$
where $D_3=-\frac{1+O(\frac{\log a}{\sqrt{a}})}{x_a^2}$ and $j=i-n_{x_a}$.
Therefore $$\int_{Q_a}je^{-\Phi_a(i)}di=\sqrt{2\pi x_a}[\frac{1}{2}+O(\frac{(\log a)^8}{\sqrt{a}})].$$
So \begin{eqnarray*}
  \sum_{i=0}^{\infty }je^{-G_a(i)}&=&\int_{i=0}^{\infty }je^{-G_a(i)}di+O(\frac{1}{a^{3/2}})\\
  &=&\int_{Q_a}je^{-G_a(i)}di+O(\frac{1}{a^{3/2}})\\
  &=&\int_{Q_a}je^{-\Phi_a(i)}(1-O(\frac{(\log a)^8}{a^2})j^2)di+O(\frac{(\log a)^{20}}{a})\\
  &=&\sqrt{2\pi x_a}[\frac{1}{2}+O(\frac{(\log a)^8}{\sqrt{a}})]+(\log a)^8
\end{eqnarray*} 
Since $\sum_{i=0}^{\infty }e^{-G_a(i)}=\sqrt{2\pi x_a}(1+O(\frac{(\log a)^8}{a}))$, 
we get the conclusion.
\end{proof}

Recall that we have defined $K_i=\frac{\sum_{i=0}^{\infty}(i-n_{x_a})^j e^{-G_a(i)}}{\sum_{i=0}^{\infty} e^{-G_a(i)}}$ and $\sigma_a=n_{x_a}-a$. Then just as proposition \ref{prop-a2},  we can prove:
\begin{prop}\label{prop-g345}
  We have
  \begin{equation}\label{e-g3}
    K_3=\frac{5}{2}x_a+O(\sqrt{a}(\log a)^8),
  \end{equation}
  \begin{equation}\label{e-g4}
    K_4=3x_a^2+O((\log a)^8a),
  \end{equation}

\end{prop}
\begin{proof}
  The proof is basically a combination of the proofs of propositions \ref{prop-j345} and \ref{prop-a2}. Denote by $j=i-n_{x_a} $.
  With the notations in the proof of corollary \ref{prop-a2}, we have 
 $$ \int_{Q_a}j^3 e^{-\Phi_a(i)}di= \sqrt{2\pi x_{a}}\frac{5}{2}x_a+O(a(\log a)^8).$$
 So we have proved equation \ref{e-g3}. 
 We can also calculate
 $$ \int_{Q_a}j^4 e^{-\Phi_a(i)}di= \sqrt{2\pi x_{a}}3x_a^2 +O(a^{3/2}(\log a)^8).$$
 And since
 $$\int_{Q_a}j^4 e^{-\Phi_a(i)}O(\frac{(\log a)^8}{a^2})j^2di=O(a^{3/2}(\log a)^8) ,$$ we have proved equation \ref{e-g4}.
\end{proof}
Recall $H_j=\frac{\sum_{i=0}^{\infty}(i-a)^j g_a(i)}{\sum_{i=0}^{\infty} g_a(i)}$ and we have 
$$\frac{d^3}{dt^3}\log f(x_a)=H_3,$$and
\begin{eqnarray*}
	\frac{d^4}{dt^4}\log f(x_a)=H_4-3H_2^2,
\end{eqnarray*}
Since $$K_2=H_2+\sigma_a^2=x_a+O((\log a)^8)$$
and
$$H_3=K_3+3\sigma_aG_2-2\sigma_a^3,$$
 we have 
$$H_3=x_a+O((\log a)^8a^{1/2}).$$
Namely, we have 
\begin{prop}\label{prop-lambda33}
  $$\frac{d^3}{dt^3}\log f(x_a)=x_a+O((\log a)^8a^{1/2}).$$
\end{prop}
And since
\begin{eqnarray*}
  H_4&=&K_4+4\sigma_a K_3+6\sigma_a^2K_2-3\sigma_a^4\\
  &=&3x^2_{a}+O(a(\log a)^8),
\end{eqnarray*}
we have 
$$H_4-3H_2^2=O(a(\log a)^8),$$
namely,
\begin{prop}\label{prop-lambda43}
  $$\frac{d^4}{dt^4}\log f(x_a)=O(a(\log a)^8).$$
\end{prop}

\section{Preparation for the ending}\label{sec-6}
From now on, for reasons that will be justified later, we redefine $$\nu(a)=\frac{h_a(x_a)}{\sqrt{x_a+\frac{1}{6}}}.$$
We also redefine $\Phi_a(t)$ with the following conditions: 
$$\begin{cases}
  \Phi_a(n_{x_a})=\phi'_a(n_{x_a})=0;\\
  \Phi_a''(i)=\frac{1}{x_i+\frac{1}{2}}.
\end{cases}$$
We also redefine 
$$M_1''(b)=\inf\{(x_a+\frac{1}{2})\lambda''(a)|a>b-\sqrt{b}\log b \} ,$$ and 
$$M_2''(b)=\sup \{(x_a+\frac{1}{2})\lambda''(a)|a>b-\sqrt{b}\log b \}, $$
hence redefining $M_1(b)$ and $M_2(b)$.
\subsection{For $\Phi_a$ and $\tilde{G}_a$.}
Now we are ready to calculate the higher derivatives of $x_a$ as a function of $a$. 

We have $$\frac{dx_a}{da}=\frac{x_a}{\frac{d^2}{dt^2}\log f(x_a)}=1+O(\frac{(\log a)^8}{a}), $$
and
\begin{eqnarray*}
  \frac{d^2x_a}{da^2}&=&\frac{dx_a}{da}(\frac{d^2}{dt^2}\log f(x_a))^{-1}-(\frac{d^2}{dt^2}\log f(x_a))^{-2}x_a \frac{d^3}{dt^3}\log f(x_a) \frac{dt_a}{da}\\
  &=&\frac{dx_a}{da}(\frac{d^2}{dt^2}\log f(x_a))^{-1}-(\frac{d^2}{dt^2}\log f(x_a))^{-3}x_a \frac{d^3}{dt^3}\log f(x_a)\\
  &=&O(\frac{(\log a)^8}{a^{3/2}}) .
\end{eqnarray*}
So, \begin{eqnarray*}
  \frac{d^3x_a}{da^3}&=&\frac{d^2x_a}{da^2}(\frac{d^2}{dt^2}\log f(x_a))^{-1}-2\frac{dx_a}{da}(\frac{d^2}{dt^2}\log f(x_a))^{-3} \frac{d^3}{dt^3}\log f(x_a) \\
  &&+3x_a(\frac{d^2}{dt^2}\log f(x_a))^{-4} (\frac{d^3}{dt^3}\log f(x_a))^2 \frac{dt_a}{da}\\
  &&-(\frac{d^2}{dt^2}\log f(x_a))^{-3}x_a \frac{d^4}{dt^4}\log f(x_a)\frac{dt_a}{da}    \\
  &=&O(\frac{(\log a)^8}{a^{3/2}})\frac{1}{x_a}-\frac{2}{x_a^2}(1+O(\frac{(\log a)^8}{\sqrt{a}}))+ \frac{3}{x_a^2}(1+O(\frac{(\log a)^8}{\sqrt{a}}))\\
  &&-\frac{1}{x_a^3}O(a(\log a)^8)\\
  &=&O(\frac{(\log a)^8}{a^{2}})
\end{eqnarray*}
Therefore \begin{eqnarray*}
  \frac{d^2x^{-1}_a}{da^2}&=&-\frac{\frac{d^2x_a}{da^2}}{x_a^2}+2\frac{(\frac{dx_a}{da})^2}{x_a^3}\\
  &=&O(\frac{\log a}{a^3})
\end{eqnarray*}

So we have the following proposition.
\begin{prop}
  For $i\in Q_a$, we have $$\Phi_a^{(3)}(i)=-\frac{1}{x_i^2}(1+O(\frac{(\log a)^8}{a})),$$
$$\Phi_a^{(4)}(i)=\frac{2}{x_i^3}+O(\frac{(\log a)^8}{a^{7/2}}),$$
and $$\Phi_a^{(5)}(i)=O(\frac{(\log a)^8}{a^4}).$$
\end{prop}
\begin{proof}
  Only the last formula needs a little explanation. We have \begin{eqnarray*}
    \Phi_a^{(5)}(i)&=&\frac{-6}{(x_i+\frac{1}{2})^4}(\frac{dx_i}{di})^3+\frac{6}{(x_i+\frac{1}{2})^3}\frac{dx_i}{di}\frac{d^2x_i}{di^2}-\frac{1}{(x_i+\frac{1}{2})^2}\frac{d^3x_i}{di^3}\\
    &=&O(\frac{(\log a)^8}{a^4})
  \end{eqnarray*}
\end{proof}
Also by proposition \ref{prop-a2}, $$x_a-x_{\tilde{a}}=\frac{1}{2}+O(\frac{(\log a)^8}{\sqrt{a}}),$$where $\tilde{a}=n_{x_a}$.
So $$\Phi_a''(\tilde{a})=x_a+O(\frac{(\log a)^8}{\sqrt{a}}).$$

Denote by $D_k=\Phi_a^{(k)}(\tilde{a})$, we can expand
$$\Phi_a(i)=\frac{D_2}{2}j^2+\frac{D_3}{6}j^3+\frac{D_4}{24}j^4+O(\frac{(\log a)^8}{a^4}j^5), $$and 
\begin{equation}\label{e-Phi-a-2}
  e^{-\Phi_a(i)}=e^{-\frac{D_2}{2}j^2}[1-\frac{D_3}{6}j^3-\frac{D_4}{24}j^4+\frac{D_3^2}{72}j^6+O(\frac{(\log a)^{13}}{a^{3/2}})] 
\end{equation}

 So \begin{eqnarray*}
  \int_{Q_a} e^{-\Phi_a(i)}di&=& \int_{Q_a}e^{-\frac{D_2}{2}j^2}[1-\frac{D_3}{6}j^3-\frac{D_4}{24}j^4+\frac{D_3^2}{72}j^6+O(\frac{(\log a)^{13}}{a^{3/2}})]di\\
  &=&(1+O(\frac{(\log a)^{13}}{a^{3/2}}))\int_{Q_a}e^{-\frac{D_2}{2}j^2}[1-\frac{D_4}{24}j^4+\frac{D_3^2}{72}j^6]di\\
  &=&(1+O(\frac{(\log a)^{13}}{a^{3/2}}))\sqrt{\frac{2\pi}{D_2}}[1-\frac{3D_4}{24D_2^2}+\frac{15D_3^2}{72D_2^3}]\\
  &=&(1+O(\frac{(\log a)^{13}}{a^{3/2}}))\sqrt{2\pi x_a}[1-\frac{1}{4x_a}+\frac{5}{24x_a}]\\
  &=&(1+O(\frac{(\log a)^{13}}{a^{3/2}}))\sqrt{2\pi x_a}[1-\frac{1}{24x_a}]
 \end{eqnarray*}

Similarly, \begin{eqnarray*}
  \int_{Q_a} j^2e^{-\Phi_a(i)}di&=& (1+O(\frac{(\log a)^{13}}{a^{3/2}}))\int_{Q_a}j^2e^{-\frac{D_2}{2}j^2}[1-\frac{D_4}{24}j^4+\frac{D_3^2}{72}j^6]di\\
  &=&(1+O(\frac{(\log a)^{13}}{a^{3/2}}))\sqrt{\frac{2\pi}{D_2}}[\frac{1}{D_2}-\frac{15D_4}{24D_2^3}+\frac{105D_3^2}{72D_2^4}]\\
  &=&(1+O(\frac{(\log a)^{13}}{a^{3/2}}))\sqrt{2\pi x_a}x_a[1-\frac{5}{4x_a}+\frac{35}{24x_a}]\\
  &=&(1+O(\frac{(\log a)^{13}}{a^{3/2}}))\sqrt{2\pi x_a}x_a[1+\frac{5}{24x_a}]
 \end{eqnarray*}
Therefore, since we now have $$\Delta G(i)=O(\frac{(\log a)^{8}}{a^{2}})(i-\tilde{a})^2,$$
\begin{eqnarray*}
  \sum j^2e^{-G_a(i)}&=& (1+O(\frac{\log a}{a^{3/2}}))\int_{Q_a}j^2e^{-G_a(i)}di\\
  &=&(1+O(\frac{\log a}{a^{3/2}}))\int_{Q_a}j^2e^{-\Phi_a(i)}(1-\Delta G(i))di\\
  &=&(1+O(\frac{(\log a)^{13}}{a^{3/2}}))[\sqrt{2\pi x_a}x_a(1+\frac{5}{24x_a})-\int_{Q_a}j^2e^{-\Phi_a(i)}\Delta G(i)di]
 \end{eqnarray*}
And similarly, $$\sum e^{-G_a(i)}=(1+O(\frac{(\log a)^{13}}{a^{3/2}}))[\sqrt{2\pi x_a}(1-\frac{1}{24x_a})-\int_{Q_a}e^{-\Phi_a(i)}\Delta G(i)di].$$
Therefore $$\frac{\sum j^2e^{-G_a(i)}}{\sum e^{-G_a(i)}}=(1+O(\frac{(\log a)^{13}}{a^{3/2}}))x_a[1+\frac{1}{4x_a}-\kappa_1+\kappa_2],$$
where $\kappa_1= \frac{1}{\sqrt{2\pi x_a}x_a}\int_{Q_a}j^2e^{-\Phi_a(i)}\Delta G(i)di$, $\kappa_2= \frac{1}{\sqrt{2\pi x_a}}\int_{Q_a}e^{-\Phi_a(i)}\Delta G(i)di$,
hence \begin{eqnarray*}
  \frac{\sum (j-a)^2e^{-G_a(i)}}{\sum e^{-G_a(i)}}&=&\frac{\sum j^2e^{-G_a(i)}}{\sum e^{-G_a(i)}}-(a-\tilde{a})^2\\
  &=&x_a[1-\kappa_1+\kappa_2+O(\frac{(\log a)^{13}}{a^{3/2}})].
\end{eqnarray*}

On the other hand, it is easy to get that 
$$\int_{Q_a} e^{-\Psi_a(i)}di=(1+\epsilon(a))\sqrt{2\pi x_a}, $$
and $$\int_{Q_a} j^2e^{-\Psi_a(i)}di=(1+\epsilon(a))x_a\sqrt{2\pi x_a}. $$

Since $a-n_{x_a}=\frac12+O(\frac{(\log a)^{8}}{\sqrt{a}})$, we also have \begin{eqnarray*}
  \int_{Q_a} (e^{-\Psi_a(i)}-e^{-\Phi_a(i)})\Delta G di&=&\int_{Q_a} e^{-\frac{1}{2x_{\tilde{a}}}j^2}O(\frac{\log a}{a^2} j^2+\frac{1}{a^2} j^3)O(\frac{\log a}{a^{3/2}} j^2)\\
  &=&O(\frac{\log a}{\sqrt{a}}),
\end{eqnarray*}
and similarly $$\int_{Q_a} j^2(e^{-\Psi_a(i)}-e^{-\Phi_a(i)})\Delta G di=O(\log a\sqrt{a}).$$

 We also have \begin{eqnarray*}
  \int_{Q_a} (e^{-\Psi_a(i)}-e^{-\Phi_a(i)})\Delta G di&=&\int_{Q_a} e^{-\frac{1}{2x_{a}}j^2}O(\frac{(\log a)}{a^{5/2}} j^2+\frac{1}{a^2} j^3)O(\frac{(\log a)^8}{a^2} j^2)\\
  &=&O(\frac{(\log a)^8}{a}),
\end{eqnarray*}
and similarly $$\int_{Q_a} j^2(e^{-\Psi_a(i)}-e^{-\Phi_a(i)})\Delta G di=O((\log a)^8).$$

Recall that $\tilde{G}_a(i)=G_a(i)-\Phi_a(i)+\Psi_a(i)$, we have
$$\int_{Q_a}e^{-\tilde{G}_a(i)}di=(1+O(\frac{(\log a)^{20}}{a^2}))\int_{Q_a}e^{-\Psi_a(i)}di-\int_{Q_a}e^{-\Psi_a(i)}\Delta G(i)di.$$
So we have the following lemma.
\begin{lem}\label{lem-redefine}
  We have $$\nu(a)=\frac{1}{\sqrt{x_a}}\int_{Q_a}e^{-\tilde{G}_a(i)}di+O(\frac{(\log a)^{13}}{a^{3/2}}).$$
\end{lem}
\begin{proof}
  We have \begin{eqnarray*}
    h_a(x_a)&=&e^{\frac{1+O((\log a)^8/a)}{2x_a}(a-\tilde{a})^2}\sum_{i=0}^{\infty} e^{-G_a(i)}\\
    &=&(1+O(\frac{(\log a)^8}{a^{3/2}}))e^{\frac{1}{8x_a}}\int_{Q_a}e^{-G_a(i)}di\\
    &=&(1+O(\frac{(\log a)^{13}}{a^{3/2}}))e^{\frac{1}{8x_a}}[\sqrt{2\pi x_a}(1-\frac{1}{24x_a})-\int_{Q_a}e^{-\Phi_a(i)}\Delta G(i)di]\\
    &=&\sqrt{2\pi x_a}(1+\frac{1}{12x_a})-\int_{Q_a}e^{-\Psi_a(i)}\Delta G(i)di+O(\frac{(\log a)^{13}}{a^{}}).
  \end{eqnarray*}
  Therefore, \begin{eqnarray*}
    \nu(a)&=&\sqrt{2\pi}-x_a^{-1/2}\int_{Q_a}e^{-\Psi_a(i)}\Delta G(i)di+O(\frac{(\log a)^{13}}{a^{3/2}})\\
    &=&\frac{1}{\sqrt{x_a}}\int_{Q_a}e^{-\tilde{G}_a(i)}di+O(\frac{(\log a)^{13}}{a^{3/2}}).
  \end{eqnarray*}

\end{proof}
Since
$$\bar{t}(\tilde{G})=\frac{(\log a)^8}{\sqrt{a}},$$
 and $$\int_{Q_a}j^2e^{-\tilde{G}_a(i)}di=(1+O(\frac{(\log a)^{20}}{a^2}))\int_{Q_a}j^2e^{-\Psi_a(i)}di-\int_{Q_a}j^2e^{-\Psi_a(i)}\Delta G(i)di,$$ we have 
\begin{eqnarray*}
  \frac{\int_{Q_a}(i-\bar{t}(\tilde{G}))^2e^{-\tilde{G}_a(i)}di}{\int_{Q_a}e^{-\tilde{G}_a(i)}di}&=&\frac{\int_{Q_a}(i-\tilde{a})^2e^{-\tilde{G}_a(i)}di}{\int_{Q_a}e^{-\tilde{G}_a(i)}di}-O(\frac{(\log a)^{16}}{a})\\
  &=&\frac{\int_{Q_a}j^2e^{-\Psi_a(i)}di-\int_{Q_a}j^2e^{-\Psi_a(i)}\Delta G(i)di}{\int_{Q_a}e^{-\Psi_a(i)}di
  -\int_{Q_a}e^{-\Psi_a(i)}\Delta G(i)di}\\&&-O(\frac{(\log a)^{16}}{a})\\
  &=&\frac{x_a\sqrt{2\pi x_a}-\int_{Q_a}j^2e^{-\Psi_a(i)}\Delta G(i)di}{\sqrt{2\pi x_a}-\int_{Q_a}e^{-\Psi_a(i)}\Delta G(i)di}-O(\frac{(\log a)^{16}}{a})\\
  &=&x_a[1-\kappa_1+\kappa_2]+O(\frac{(\log a)^{16}}{a})
\end{eqnarray*}
Therefore, we have proved the following proposition.
\begin{prop}\label{prop-d2-logf}
  We have \begin{equation}
    \frac{d^2}{dt^2}\log f(x_a)=(1+O(\frac{(\log a)^{13}}{a^{3/2}}))\frac{\int_{Q_a}(i-\bar{t}(\tilde{G}))^2e^{-\tilde{G}_a(i)}di}{\int_{Q_a}e^{-\tilde{G}_a(i)}di}
  \end{equation}
\end{prop}

We also have for $i\in Q_a$,
$$x_a^{-1}+\frac{M_1-1}{x_a-2\sqrt{a}\log a}\leq \tilde{G}_a''(i)\leq x^{-1}_{a}+ \frac{M_2-1}{x_a-2\sqrt{a}\log a},$$
and that $\exists C_0>0$ so that the ratio of the upper and lower bounds satisfies
\begin{equation}\label{e-use-xi-2}
  \frac{x^{-1}_{a}+(M_2-1)(x_a-2\sqrt{a}\log a)^{-1}}{x^{-1}_{a}+(M_1-1)(x_a-2\sqrt{a}\log a)^{-1}}<\frac{M_2}{M_1}(1+C_0\frac{(\log a)^9}{a^{3/2}}).
\end{equation}

\

\subsection{For $\phi_a$ and $\tilde{g}_a$.}
Now we turn to the part of replacing $g_a$ with $\tilde{g}_a$. This part is basically a repetition of that part in subsection \ref{subsec-fls} with smaller $\Delta g$. Now we have $$\Delta g(t)=O((\log a)^8\tau^2),$$where $\tau=t-t_{a+1}$. So $$e^{-g_a(t)}=e^{-\phi_a(t)}(1-\Delta g+O(\frac{(\log a)^{20}}{a^2})),$$
and $$e^{-\tilde{g}_a(t)}=e^{-\psi_a(t)}(1-\Delta g+O(\frac{(\log a)^{20}}{a^2})).$$
Also $$\int  e^{-\phi_a(t)}\Delta_g-\int \tau^2 e^{-\psi_a(t)}\Delta_g=O(\frac{(\log a)^8}{a^{2}}),$$ and 
$$\int \tau^2 e^{-\phi_a(t)}\Delta_g-\int \tau^2 e^{-\psi_a(t)}\Delta_g=O(\frac{(\log a)^8}{a^{3}}).$$

Like the calculations for $\Phi_a$, since $\phi_a^{(k)}(t)=e^t$ for $k\geq 2$, we can calculate
$$\int_{P_a}e^{-\phi_a(t)}dt=\sqrt{\frac{2\pi}{x_{a+1}}}(1+\frac{1}{12x_{a+1}}+O(\frac{1}{a^2})), $$
and $$\int_{P_a}\tau^2 e^{-\phi_a(t)}dt=\frac{1}{x_{a+1}}\sqrt{\frac{2\pi}{x_{a+1}}}(1+\frac{5}{6x_{a+1}}+O(\frac{1}{a^2})). $$
\begin{lem}\label{lem-redefine2}
  We have $$\nu(a)=\sqrt{x_{a+1}}\int_{P_a}e^{-\tilde{g}_a(t)}dt+O(\frac{(\log a)^{13}}{a^{3/2}}).$$
\end{lem}
\begin{proof}
Recall that $$\frac{h_a(x_{a+1})}{x_{a+1}}=\int_{-\infty}^{\infty }e^{-g_a(t)}dt.$$
Since $x_{a+1}=x_a+1+O(\frac{(\log a)^8}{a})$, we have
$$h_a(x_{a+1})=e^{\frac{x_a}{2}(\frac{1}{x_a})^2(1+O(\frac{(\log a)^8}{a}))}h_a(x_a).$$
Therefore $$h_a(x_a)=(1+\frac{1}{2x_a}+O(\frac{(\log a)^8}{a^2}))x_a\int_{-\infty}^{\infty }e^{-g_a(t)}dt.$$
Since $\sqrt{\frac{2\pi}{x_{a+1}}}=\sqrt{\frac{2\pi}{x_{a}}}(1-\frac{1}{2x_a}+O(\frac{1}{a^2}))$, 
$$\int_{P_a}e^{-\phi_a(t)}dt=\sqrt{\frac{2\pi}{x_{a}}}(1-\frac{5}{12x_{a}}+O(\frac{1}{a^2})). $$

On the other hand, $$\int_{P_a}e^{-\psi_a(t)}dt=\sqrt{\frac{2\pi}{x_{a+1}}}(1+\epsilon(a)). $$
  The rest is similar to the proof of lemma \ref{lem-redefine}. 
\end{proof}

\begin{prop}\label{prop-d2-lambda}
  We have \begin{equation}
    \lambda''(a)=(1+\frac{1}{2x_{a+1}}+O(\frac{(\log a)^{13}}{a^{3/2}}))\frac{\int_{P_a}(t-\bar{t}(\tilde{g}))^2e^{-\tilde{g}_a(t)}dt}{\int_{P_a}e^{-\tilde{g}_a(t)}dt}
  \end{equation}
\end{prop}
\begin{proof}
  We only need to notice that 
  \begin{eqnarray*}
    \frac{\int_{P_a}\tau^2 e^{-\phi_a(t)}dt}{\int_{P_a}e^{-\phi_a(t)}dt}&=&\frac{1}{x_{a+1}}(1+\frac{3}{4x_{a+1}}+O(\frac{1}{a^2}))\\
  \end{eqnarray*}
  So $$\frac{\int_{P_a}\tau^2 e^{-\phi_a(t)}dt}{\int_{P_a}e^{-\phi_a(t)}dt}-\frac{1}{4x_{a+1}^2}=\frac{1}{x_{a+1}}(1+\frac{1}{2x_{a+1}}+O(\frac{1}{a^2})).$$
  On the other hand $$\frac{\int_{P_a}\tau^2 e^{-\psi_a(t)}dt}{\int_{P_a}e^{-\psi_a(t)}dt}=\frac{1}{x_{a+1}}+\epsilon(a).$$
  The rest is similar to the proof of proposition \ref{prop-d2-logf}.
\end{proof}
\begin{rem}
  The term $\frac{1}{2x_{a+1}}$ explains our new definition of $\Phi_a$ that makes $\Phi_a''(i)=\frac{1}{x_i+1/2}$.
\end{rem}

We also have for $t\in P_a$,
$$x_{a+1}+(M_1-1)(x_a+\sqrt{a}\log a)\leq \tilde{g}_a''(t)\leq x_{a+1}+ (M_2-1)(x_a+\sqrt{a}\log a),$$
and that $\exists C_1>0$ so that the ratio of the upper and lower bounds satisfies
\begin{equation}\label{e-use-xi-3}
  \frac{x_{a+1}+(M_2-1)(x_a+\sqrt{a}\log a)}{x_{a+1}+(M_1-1)(x_a+\sqrt{a}\log a)}<\frac{M_2}{M_1}(1+C_1\frac{(\log a)^9}{a^{3/2}}).
\end{equation}
Let $B_0=\max\{C_0,C_1\}$, we now define 
\begin{equation}\label{e-B0}
  m(a)=\frac{M_2}{M_1}(1+B_0\frac{(\log a)^9}{a^{3/2}}).
\end{equation}

\section{The final gap}

In this subsection, we prove the following theorem 
\begin{theorem}\label{theo-general-a32}
  $\exists C>0$ such that for $a$ large enough, we have 
  $$m(a)\leq 1+C\frac{(\log a)^{13}}{a^{3/2}}.$$
\end{theorem}
The proof is basically a repetition of the proof of theorem \ref{theo-general-log}. But given the subtlety of the error, we go over the proof again.
We have, by corollary \ref{cor-tt-2},  \begin{eqnarray*}
  \nu'(b)&=&\tilde{t}_b-t_b-\frac{1}{2\frac{d^2}{dt^2}\log f(x_b)}\frac{x_b}{x_b+\frac16}\\,
  &=&O(\frac{(\log b)^{8}}{b^{3/2}})
\end{eqnarray*}
and
\begin{eqnarray*}
  \nu''(b)&=&\lambda''(b)-\frac{1}{\frac{d^2}{dt^2}\log f(x_b)}+\frac{\frac{d^3}{dt^3}\log f(x_b)}{2(\frac{d^2}{dt^2}\log f(x_b))^3}\frac{x_b}{x_b+\frac12}\\
  &&+\frac{1}{2\frac{d^2}{dt^2}\log f(x_b)}\frac{dx_b/db}{6(x_b+\frac16)^2}\\
  &=&\lambda''(b)-\frac{1}{\frac{d^2}{dt^2}\log f(x_b)}+\frac{1}{2x_b^2}+O(\frac{(\log a)^8}{b^{5/2}})\\
  &=&(\frac{d^2}{dt^2}\log f(x_b))^{-1}[(\lambda''(b)+\frac{1}{2x_b^2})\frac{d^2}{dt^2}\log f(x_b)-1]+O(\frac{(\log a)^8}{b^{5/2}})
\end{eqnarray*}

\subsection*{the non-monotonic case}
We first assume that $\nexists n>0$ such that $\nu(a)$ is monotone for $a>n$. 

Let $b$ be a local minimum point of $\nu(a)$. By assumption, there are infinite such $b$, so our following notations $O(\frac{1}{\sqrt{b}})$, etc. make sense. Then we have 

$$\tilde{t}_b-t_b=\frac{1}{2\frac{d^2}{dt^2}\log f(x_b)},$$
and 
$$(\lambda''(b)+\frac{1}{2x_b^2})\frac{d^2}{dt^2}\log f(x_b)\geq 1+O(\frac{(\log a)^8}{b^{3/2}}).$$
Recall that we denote by $$\gamma(a)=\frac{\int_{P_a}(t-\bar{t}(\tilde{g}))^2e^{-\tilde{g}_a(t)}dt}{\int_{P_a}e^{-\tilde{g}_a(t)}dt},$$ and 
$$\Gamma(a)=\frac{\int_{Q_a}(i-\bar{t}(\tilde{G}))^2e^{-\tilde{G}_a(i)}di}{\int_{Q_a}e^{-\tilde{G}_a(i)}di}, $$
then by propositions \ref{prop-d2-logf} and \ref{prop-d2-lambda}, we get
$$(1+\frac{1}{x_b})\gamma(b)\Gamma(b) \geq 1+O(\frac{(\log b)^8}{b^{3/2}}),$$ 
since $\gamma(b)=\frac{1}{x_{b+1}}(1+O(\frac{(\log b)^8}{b}))$.
Then let $p_1=\tilde{p}(m(b))$, $q_1=\tilde{q}(m(b))$, by lemma \ref{lem-redefine} and lemma \ref{lem-redefine2}, we have 
\begin{eqnarray*}
  \gamma(b)&\leq& p_1(\int_{P_a}e^{-\tilde{g}_a(t)}dt+\epsilon(b))^2\\
  &=&p_1(\frac{\nu(b)}{\sqrt{x_{b+1}}}+O(\frac{(\log a)^{13}}{a^2}))^2\\
  &=&p_1\frac{\nu^2(b)}{x_{b+1}}(1+O(\frac{(\log a)^{13}}{a^{3/2}})),
\end{eqnarray*}
and similarly $$\Gamma(b)\leq p_1 x_b\nu^2(b)(1+O(\frac{(\log a)^{13}}{a^{3/2}})).$$
Therefore, $$p_1^2\nu^4(b)\geq 1+O(\frac{(\log a)^{13}}{a^{3/2}}),$$
namely $$\nu(b)\geq \frac{1}{\sqrt{p_1}}(1+O(\frac{(\log b)^{13}}{b^{3/2}})).$$
Similarly when $b$ be a local maximum point of $\nu(a)$, we have 
$$\nu(b)\leq \frac{1}{\sqrt{q_1}}(1+O(\frac{(\log b)^{13}}{b^{3/2}})).$$
Then by the same argument preceeding lemma \ref{lem-p1q1} we get that there exists $C_1>0$ such that 
when $b$ is a local extremal point of $\nu(a)$, then 
$$(1-C_1\frac{(\log b)^{13}}{b^{3/2}})\frac{1}{\sqrt{p_1}}<\nu(a)<(1+C_1\frac{(\log b)^{13}}{b^{3/2}})\frac{1}{\sqrt{q_1}},$$
for $a\geq b$.
Then we can prove a finer version of lemma \ref{lem-p1q12} by repeating its proof. 
\begin{lem}\label{lem-p1q13}
  There exists $C>0$ such that for $a$ large enough, assume $m(a)>1+C\frac{\log a}{\sqrt{a}}$, then
  $$(1-C_1\frac{(\log a)^{13}}{a^{3/2}})\frac{1}{\sqrt{p_1}}< \nu(a')< (1+C_1\frac{(\log a)^{13}}{a^{3/2}})\frac{1}{\sqrt{q_1}},$$
  where $p_1=\tilde{p}(m(a))$ and $q_1=\tilde{q}(m(a))$,
  for all $a'\geq a+2\sqrt{a}\log a$.
\end{lem}
\begin{proof}
  We have proved the case when $a$ is a local extremal point of $\nu$. Therefore if there is a local extremal point of $\nu$ on the interval $[a,a+2\sqrt{a}\log a]$, the conclusion of the proposition also follows. So we only need to consider the case when $\mu$ is monotone on $[a,a+2\sqrt{a}\log a]$. We assume that $\mu$ is decreasing. Then $\exists a_1>a+2\sqrt{a}\log a$ a local minimum point of $\nu$ such that $\nu(a_1)\leq \nu(a')$ for all $a'\in [a,a_1]$. It then follows that we have 
  $$\nu(a')\geq (1-C_1\frac{(\log b)^{13}}{b^{3/2}})\frac{1}{\sqrt{p_1}},$$ for all $a'\geq a$. If $\nu(a+2\sqrt{a}\log a)<(1+C_1\frac{(\log a)^{13}}{a^{3/2}})\frac{1}{\sqrt{q_1}}$, then the upper bound for $a'\geq a+2\sqrt{a}\log a$ also follows. So we assume that $\nu(a+2\sqrt{a}\log a)\geq (1+C_1\frac{(\log a)^{13}}{a^{3/2}})\frac{1}{\sqrt{q_1}}$, hence $$\nu(a')\geq (1+C_1\frac{(\log a)^{13}}{a^{3/2}})\frac{1}{\sqrt{q_1}},$$ for $a'\in [a,a+2\sqrt{a}\log a]$.

  Notice that for all $b\geq a$ we have
  \begin{eqnarray*}
    \gamma(b)\geq q_1\frac{\nu^2(b)}{x_{b+1}}(1+O(\frac{(\log a)^{13}}{a^{3/2}}))
  \end{eqnarray*}
  
  Therefore 
  $$x_{b+1}\gamma(b)\geq 1+C_4\frac{(\log a)^{13}}{a^{3/2}},$$
  for $b\in [a,a+2\sqrt{a}\log a]$, for some $C_4$ independent of $a$.
  By proposition \ref{prop-d2-lambda}, we have $$x_b\lambda''(b)\geq 1-\frac{1}{2x_b}+C_5\frac{(\log a)^{13}}{a^{3/2}},$$
  for $b\in [a,a+2\sqrt{a}\log a]$, for some $C_5$ independent of $a$.
  Let $a_2=a+\sqrt{a}\log a$, then 
 \begin{eqnarray*}
  h_{a_2}(x_{a_2})&\leq& (1+C_6\frac{(\log a)^{13}}{a^{3/2}})e^{\frac{1}{8x_{a_2}}}\int_{Q_a}e^{-\Phi_{a_2}(i)}di\\
  &\leq & (1+C_7\frac{(\log a)^{13}}{a^{3/2}})\sqrt{2\pi x_{a_2}}(1+\frac{1}{12x_{a_2}}),
 \end{eqnarray*}
 for some $C_6, C_7$ independent of $a$. 
  So we have 
  $$\nu(a_2)\leq (1+C_7\frac{(\log a)^{13}}{a^{3/2}})\sqrt{2\pi}.$$
 We get 
 \begin{equation}\label{e-impossible}
  (1+C_1\frac{(\log a)^{13}}{a^{3/2}})\frac{1}{\sqrt{q_1}}\leq (1+C_7\frac{(\log a)^{13}}{a^{3/2}})\sqrt{2\pi}.
 \end{equation}

Then one can see that there exists $C>0$ such that if $$m(a)>1+C\frac{(\log a)^{13}}{a^{3/2}} ,$$ then formula \ref{e-impossible} is impossible.

The case when $\mu$ is increasing is similar.

\end{proof}

\begin{theorem}
  There exists $C>0$ such that for $a$ large enough such that $$m(a)<1+C\frac{(\log a)^{13}}{a^{3/2}}.$$
\end{theorem}
\begin{proof}
  
Let $p_1=\tilde{p}(m(a))$ and $q_1=\tilde{q}(m(a))$, then we get 
$$(1-C_1\frac{(\log a)^{13}}{a^{3/2}})\frac{1}{\sqrt{p_1}}< \nu(a')< (1+C_1\frac{(\log a)^{13}}{a^{3/2}})\frac{1}{\sqrt{q_1}},$$
for all $a'\geq a_1=a+2\sqrt{a}\log a$. Since 
$$q_1\nu^2(b)(1-C_2\frac{(\log a)^{13}}{a^{3/2}})\leq x_{b+1}\gamma(b)\leq p_1\nu^2(b)(1+C_2\frac{(\log a)^{13}}{a^{3/2}}), $$
 and
 \begin{equation*}
  q_1\nu^2(b)(1-C_3\frac{(\log a)^{13}}{a^{3/2}})\leq \frac{\Gamma(b)}{x_b}\leq p_1\nu^2(b)(1+C_3\frac{(\log a)^{13}}{a^{3/2}})\frac{1}{\sqrt{q_1}},
 \end{equation*}
for $b\geq a_1$, for some $C_2,C_3$ independent of $a$, we get
\begin{equation*}
  \frac{q_1}{p_1}(1-C_4\frac{(\log a)^{13}}{a^{3/2}})\leq x_{b+1}\gamma(b)\leq \frac{p_1}{q_1}(1+C_4\frac{(\log a)^{13}}{a^{3/2}}),
 \end{equation*}
 and
 \begin{equation*}
  \frac{q_1}{p_1}(1-C_5\frac{(\log a)^{13}}{a^{3/2}})\leq \frac{\Gamma(b)}{x_b}\leq \frac{p_1}{q_1}(1+C_5\frac{(\log a)^{13}}{a^{3/2}}),
 \end{equation*}
for $b\geq a_1$, for some $C_4,C_5$ independent of $a$. 
Therefore \begin{equation*}
  \frac{q_1}{p_1}(1-C_6\frac{(\log a)^{13}}{a^{3/2}})\leq (x_b+\frac12)\lambda''(b)\leq \frac{p_1}{q_1}(1+C_6\frac{(\log a)^{13}}{a^{3/2}}),
 \end{equation*}
 and
 \begin{equation*}
  \frac{q_1}{p_1}(1-C_6\frac{(\log a)^{13}}{a^{3/2}})\leq \frac{1}{x_b}\frac{d^2}{dt^2}\log f(x_b)\leq \frac{p_1}{q_1}(1+C_6\frac{(\log a)^{13}}{a^{3/2}}),
 \end{equation*}
for $b\geq a_1$, for some $C_6$ independent of $a$

So let $a_1'$ satisfy $a_1'-\log a_1'\sqrt{a_1'}=a_1$, then $$M_1(a_1')>\frac{q_1}{p_1}(1-C_6\frac{(\log a)^{13}}{a^{3/2}})$$ and $$M_2(a_1')< \frac{p_1}{q_1}(1+C_6\frac{(\log a)^{13}}{a^{3/2}}).$$
 Therefore, \begin{eqnarray*}
  m(a_1')<\frac{p^2_1}{q^2_1}(1+3C_6\frac{(\log a)^{13}}{a^{3/2}}) .
 \end{eqnarray*}
Then let $C_0$ be the constant $C$ in the statement of lemma \ref{lem-p1q13}.

 We claim that for $a$ large enough we have $$m(a)<1+(4C_0+120C_6)\frac{(\log a)^{13}}{a^{3/2}}.$$
It is enough to show that for $a$ large engough, $$m(2a)\leq 1+(2C_0+60C_6)\frac{(\log a)^{13}}{a^{3/2}}.$$ 
Assume this is not the case, then $m(a')>1+(2C_0+60C_6)\frac{(\log a)^{13}}{a^{3/2}},$
for $a'\in [a,2a]$. Therefore $$m(a')>1+(C_0+42C_6)\frac{(\log a')^{13}}{a^{'3/2}},$$ for $a'\in [a,2a]$. 
So we can apply the conclusion of lemma \ref{lem-p1q13}. 
\begin{eqnarray*}
  m(a_1')-1&<&0.9m(a)(1+3C_6\frac{(\log a)^{13}}{a^{3/2}})+3C_6\frac{(\log a)^{13}}{a^{3/2}}\\
  &<&0.95(m(a)-1).
\end{eqnarray*}
 
 Then for $n\geq 1$ we define $a_{n+1}=a_n'+2\sqrt{a_n'}\log a_n'$ and $a_{n+1}'$ accordingly and inductively. And by induction, we get 
 $$m(a_n')-1<0.95^{n-1}(m(a)-1).$$
 It is easy to see that when $n=\frac{\sqrt{a}}{4\log a}+1$, $a_{n}'<2a$. But $$0.95^{\frac{\sqrt{a}}{4\log a}}<\frac{(\log a)^{12}}{a^{3/2}},$$
 So we have arrived at a contraction.
\end{proof}
\subsection*{the general case }

\begin{theorem}\label{theo-general-log-3}
  There exists $C>0$ such that for $a$ large enough such that $$m(a)<1+C\frac{(\log a)^{13}}{a^{3/2}}.$$
\end{theorem}
\begin{proof}[Proof of theorem \ref{theo-general-a32}]
  The non-monotonic case has been proved in theorem \ref{theo-non-monotone}. If $\nu(a)$ is monotone, we assume, for example, 
  that $\nu'(a)\geq 0$ for $a$ large enough.
  
  We have seen that $\nu'(a)=O(\frac{(\log b)^{8}}{b^{3/2}})$. 
  There are two possibilities:
  \begin{itemize}
    \item  $\nu'(a)$ is monotone, then there is a sequence $a_j\to\infty$ so that $\nu''(a_j)=O(\frac{(\log b)^{8}}{b^{5/2}})$; 
    \item $\nu'(a)$ is not monotone, then around a local minimum where $|\nu'(a)|<\frac{1}{a}$, we have $\nu''(a)=0$.
  \end{itemize}
  In both cases, we can apply the argument for the non-monotonic case to get the claimed estimates in the theorem.
  The case $\nu'(a)< 0$ is similar.
\end{proof}
The following corollary then follows.
\begin{cor}\label{cor-final}
  We have 
  $$\frac{d^2}{dt^2}\log f(x)=x+O(\frac{(\log x)^{13}}{x^{1/2}}),$$
  and $$(x_a+\frac12)\lambda''(a)=1+O(\frac{(\log x)^{13}}{x^{3/2}}).$$
\end{cor}
So $$\frac{du}{dx}=\frac{1}{x}\frac{d^2}{dt^2}\log f(x)=1+O(\frac{(\log x)^{13}}{x^{3/2}}).$$
Therefore $$\omega(\beta)=\lim_{x\to \infty}(u-x) \textit{ exists}.$$
Then by L'Hospital's rule, we get theorem \ref{thm-main}. Moreover, we have 
$$u-x-\omega(\beta)=O(\frac{(\log x)^{13}}{x^{1/2}}).$$
So $$(\log f(x))'=1+\frac{\omega(\beta)}{x}+O(\frac{(\log x)^{13}}{x^{3/2}}),$$
which implies that $$C_\beta=\lim_{x\to\infty }\frac{f(x)}{x^{\omega(\beta)}e^x}  \textit{ exists}.$$

\

Now we come to the proof of theorem \ref{thm-second}. Consider a pair of numbers $0\leq \beta_1<\beta_2<1$. We make a normalization so that $$f_{\beta_1}(0)=f_{\beta_2}(0)=1.$$ We denote by $c(i,\beta_j)$ the coeffient function for $f_{\beta_j}$, $j=1,2$. 
Consider the set $$Q=\{\frac{c(i,\beta_2)}{c(i,\beta_1)}|i\in \R,i\geq 1 \}.$$
Since $\int_{0}^{\infty}\frac{1}{f_{\beta_1}}dx>\int_{0}^{\infty}\frac{1}{f_{\beta_2}}dx$, we have $\sup Q>1$.

A short argument at the beginning of section 4 in \cite{sun-sun-2} shows that the condition $$\int_{0}^{\infty }\frac{c(i,\beta_1)}{f_{\beta_1}}dx=\int_{0}^{\infty }\frac{c(i,\beta_2)}{f_{\beta_2}}dx=1,$$
implies the following
\begin{equation}\label{e-cb1}
	\frac{c(i,\beta_2)}{c(i,\beta_1)}<\int_{0}^{\infty}\frac{c(i,\beta_1)x^i}{f_{\beta_1}}\frac{f_{\beta_2}}{f_{\beta_1}}dx,
	\end{equation}
  and
	\begin{equation}\label{e-cb2}
    \frac{c(i,\beta_1)}{c(i,\beta_2)}<\int_{0}^{\infty}\frac{c(i,\beta_2)x^i}{f_{\beta_2}}\frac{f_{\beta_1}}{f_{\beta_2}}dx.
	\end{equation} 	
	The first formula implies that $\sup Q$ cannot be attained at finite $i$. Let $k(i)=\frac{c(i,\beta_2)}{c(i,\beta_1)}$, then the main arguments in section 4 in \cite{sun-sun-2} actually shows the following:
  \begin{lem}\label{lem-last}
    We have $$\liminf_{i\to\infty }k(i)=\limsup_{i\to\infty }k(i).$$
  \end{lem}
  Using corollary \ref{cor-final}, we can give a shorter proof. 
 \begin{proof}
  Denote by $\delta_j=\lim_{x\to \infty }(u-x)$, $j=1,2$ and denote by $x_{j,a}$ the $x_a$ correspongding to $f_{\beta_j}$. Then $$a=x_{a_j}+\delta_j+O(\frac{(\log a)^{13}}{x^{1/2}}).$$
  So \begin{eqnarray*}
    \frac{d^2}{di^2}\log k(i)&=&\frac{1}{a-\delta_1+\frac{1}{2}}-\frac{1}{a-\delta_2+\frac{1}{2}}+O(\frac{(\log a)^{13}}{x^{5/2}})\\
    &=&\frac{\delta_2-\delta_1}{a^2}+O(\frac{(\log a)^{13}}{x^{5/2}})\\
  \end{eqnarray*}  
  Since $\lambda'(a)-\log a\to 0$ as $a\to \infty$, we have $$\frac{d}{di}\log k(i)=\frac{\delta_2-\delta_1}{a}+O(\frac{(\log a)^{13}}{x^{3/2}}).$$
  No matter the relation between $\delta_1$ and $\delta_2$, we have $$\liminf_{i\to\infty }k(i)=\limsup_{i\to\infty }k(i).$$
 \end{proof}

So $\inf Q$ is attained at some $i_0$. Assume $k_{i_0}\leq 1$, then formula \ref{e-cb2} gives us a contraction. Therefore, we must have $k_i>1$ for all $i\geq 1$. As a direct corollary, we get that $\omega(\beta)$ is increasing. 

To prove the upper bound, we replace the pair $(f_{\beta_1},f_{\beta_2})$ with $(f_{\beta},xe^x)$ in the preceeding argument. Write $f(x)=xe^x=\sum_{i=0}^{\infty}c(i) x^i$. Then since $f(x)$ also satisfies the condition that $\int_{i=0}^{\infty}\frac{c(i)x^i}{f(x)}dx=1$ for $i\geq 1$,
we consider again the set $$Q=\{\frac{c(i,\beta)}{c(i)}|i\in \R,i\geq 1 \}.$$
We renormalize $f_\beta$ so that $c(1,\beta)=1$. Then since $\int_{i=0}^{\infty}\frac{x}{f(x)}dx=\int_{i=0}^{\infty}\frac{x}{f_\beta(x)}dx$ and $c(0)=0$, we must have $\sup Q>1$. Then formula \ref{e-cb1} implies that $\sup Q$ cannot be obtained at finite $i$.
And the proof of lemma \ref{lem-last} can be repeated to show that $$\liminf_{i\to\infty }k(i)=\limsup_{i\to\infty }k(i)=\sup Q>1,$$ 
which implies that $$\omega(\beta)\leq \lim_{x\to \infty}\frac{\log (xe^x)-x}{\log x}=1.$$
\bibliographystyle{plain}

\bibliography{references}

\end{document}